\documentclass[a4paper,reqno,10pt]{amsart}
\makeatletter
\newcommand*{\rom}[1]{\expandafter\@slowromancap\romannumeral #1@}
\makeatother
\usepackage{epsf,graphicx,epsfig}
\usepackage{amscd}
\usepackage{amsmath,latexsym,amssymb,amsthm}
\usepackage[nospace,noadjust]{cite}
\usepackage{textcomp}
\usepackage{setspace,cite}
\usepackage{mathrsfs,amsmath}
\usepackage{chemarrow}
\usepackage{lscape,fancyhdr,fancybox}
\usepackage{stmaryrd}
\usepackage[all,cmtip]{xy}
\usepackage{tikz-cd}
\usepackage{cancel}
\usetikzlibrary{shapes,arrows,decorations.markings}
\usepackage{graphicx}

\usepackage{color}
\usepackage{bbm, dsfont}
\usepackage{xcolor}
\usepackage{url}
\usepackage{enumerate}
\usepackage[mathscr]{euscript}
  \theoremstyle{plain}
\swapnumbers
    \newtheorem{thm}{Theorem}[section]
    \newtheorem{proposition}[thm]{Proposition}

    \newtheorem{subsec}[thm]{}
\theoremstyle{definition}
    \newtheorem{definition}[thm]{Definition}
        \newtheorem{remark}[thm]{Remark}
    \newtheorem{exam}[thm]{Example}

\theoremstyle{remark}

\title{}
\author{}
\date{}
\usepackage{amssymb}

\usepackage{hyperref}
\hypersetup{
colorlinks,
citecolor=blue,
filecolor=black,
linkcolor=blue,
urlcolor=black
}

\begin{document}

\title[Higher structures for Lie $H$-pseudoalgebras]{Higher structures for Lie $H$-pseudoalgebras}

\author{Apurba Das}
\address{Department of Mathematics,
Indian Institute of Technology, Kharagpur 721302, West Bengal, India.}
\email{apurbadas348@gmail.com, apurbadas348@maths.iitkgp.ac.in}

\maketitle

\begin{abstract}
    Let $H$ be a cocommutative Hopf algebra. The notion of Lie $H$-pseudoalgebra is a multivariable generalization of Lie conformal algebras. In this paper, we study some higher structures related to Lie $H$-pseudoalgebras where we increase the flexibility of the Jacobi identity. Namely, we first introduce $L_\infty$ $H$-pseudoalgebras (also called strongly homotopy Lie $H$-pseudoalgebras) as the homotopy analogue of Lie $H$-pseudoalgebras. We give several equivalent descriptions of such homotopy algebras and show that some particular classes of these homotopy algebras are closely related to the cohomology of Lie $H$-pseudoalgebras and crossed modules of Lie $H$-pseudoalgebras. Next, we introduce another higher structure, called Lie-$2$ $H$-pseudoalgebras which are the categorification of Lie $H$-pseudoalgebras. Finally, we show that the category of Lie-$2$ $H$-pseudoalgebras is equivalent to the category of certain $L_\infty$ $H$-pseudoalgebras.
    
    %such as strongly homotopy Lie $H$-pseudoalgebras (also called $L_\infty$-pseudoalgebras) and Lie $2$-pseudoalgebras by increasing the flexibility of the Jacobi identity of a Lie $H$-pseudoalgebra. We show that some particular classes of $L_\infty$-pseudoalgebras are closely related to the cohomology of Lie $H$-pseudoalgebras and crossed modules of Lie $H$-pseudoalgebras. Finally, we prove that the category of Lie $2$-pseudoalgebras is equivalent to the category of certain $L_\infty$-pseudoalgebras.

    %We first introduce strongly homotopy Lie $H$-pseudoalgebras (also called $L_\infty$-pseudoalgebras) as the homotopy analogue of Lie $H$-pseudoalgebras. We give several equivalent descriptions of such homotopy algebras and show that some particular classes of these homotopy algebras are closely related to the cohomology of Lie $H$-pseudoalgebras and crossed modules of Lie $H$-pseudoalgebras. In the end, we introduce another higher structure, called Lie $2$-pseudoalgebras which are the categorification of Lie pseudoalgebras. Finally, we show that the category of Lie $2$-pseudoalgebras is equivalent to the category of certain $L_\infty$-pseudoalgebras.
\end{abstract}

\medskip

\medskip

\medskip



\tableofcontents

\section{Introduction}

\subsection{Lie {\em H}-pseudoalgebras}
The concept of Lie conformal algebras was introduced by Kac \cite{kac, kac2} as a tool to study vertex operator algebras. Lie conformal algebras encode the axiomatic description of the operator product expansion of chiral fields in conformal field theory. Subsequently, the notion of Lie $H$-pseudoalgebra was introduced by Bakalov, D' Andrea and Kac \cite{bakalov-andrea-kac} as a multivariable generalization of Lie conformal algebras. While Lie conformal algebras are defined on ${\bf k}[\partial]$-modules, Lie $H$-pseudoalgebras are defined on left $H$-modules. Roughly, given a cocommutative Hopf algebra $H$, one first considers a pseudotensor category $\mathcal{M}^* (H)$. A Lie $H$-pseudoalgebra is simply a Lie algebra object in $\mathcal{M}^* (H)$. They have close connections to differential Lie algebras of Ritt \cite{ritt} and Hamiltonian formalism in the theory of nonlinear evolution equations \cite{gelfand-dorfman}. Given an ordinary Lie algebra $\mathfrak{g}$, one can construct a Lie $H$-pseudoalgebra structure on the tensor product $H \otimes \mathfrak{g}$. This is called the current Lie $H$-pseudoalgebra of $\mathfrak{g}$ and it is denoted by $\mathrm{Cur}(\mathfrak{g})$. On the other hand, each element in a Lie $H$-pseudoalgebra may be recast in terms of its Fourier coefficients, sometimes called creation and annihilation operators in the physical literature. The space of all such annihilation operators forms a Lie algebra (in general infinite-dimensional). Therefore, Lie $H$-pseudoalgebras are also useful to study infinite-dimensional Lie algebras. Structure theory, representations and cohomology of Lie $H$-pseudoalgebras are widely studied in \cite{bakalov-andrea-kac,bakalov-andrea-kac2,bakalov-andrea-kac3,andrea-marchei,das,wu} and in the references therein. See \cite{liberati,gole,retakh,wu-g,wu2} for some other types of pseudoalgebras and relevant results.

%When $H = U (\mathfrak{g})$ the universal enveloping algebra of a commutative Lie algebra $\mathfrak{g}$, then Lie $H$-pseudoalgebras encode the operator product expansion of ultralocal fields.

\subsection{Strongly homotopy algebras and categorifications}
In mathematics, when one allows to increase the flexibility of a structure, one obtains higher structures. This can be done often in two ways, namely, by homotopification and categorification of given algebraic identity/identities. Higher structures, such as strongly homotopy algebras and categorifications of algebras are key objects in higher differential geometry, higher gauge theory and infinity-category theory \cite{baez-crans,rogers,cheng,lada-s}. The notion of an $L_\infty$ algebra (also called a strongly homotopy Lie algebra) which is the homotopification of a Lie algebra, was introduced by Lada and Stasheff \cite{lada-s} and further studied by Lada and Markl \cite{lada-markl}. This homotopy algebra plays a significant role in many areas of mathematics and mathematical physics including deformation theory, quantization of Poisson manifolds and classical field theory \cite{kont,rogers,fre}. On the other hand, the notion of a Lie-$2$ algebra was introduced by Baez and Crans \cite{baez-crans} as the categorification of Lie algebras. Like Lie algebras are closely related to the set-theoretical solutions of the Yang-Baxter equation, Lie-$2$ algebras are related to the Zamolodchikov tetrahedron equation. Although strongly homotopy algebras and categorifications of algebras are two different generalizations of algebras, it has been observed in \cite{baez-crans} that these higher structures are closely connected. More precisely, they showed that the category of $2$-term $L_\infty$ algebras is equivalent to the category of Lie-$2$ algebras. Among other results, they considered and characterized skeletal and strict $L_\infty$ algebras.

\subsection{Graded Lie {\em H}-pseudoalgebras and strongly homotopy Lie {\em H}-pseudoalgebras} The concept of graded Lie algebras are the Lie objects in the category of graded vector spaces. In \cite{wu-g} the author has considered graded Lie $H$-pseudoalgebras in his study of generalizations of some results of Lie $H$-pseudoalgebras in the graded context. Subsequently, representations and cohomology of graded Lie $H$-pseudoalgebras are studied in \cite{sun}. The author also generalized some results of Lie $H$-pseudobialgebras and Manin triples of Lie $H$-pseudoalgebras in the graded context.

In this paper, we aim to study two higher structures in the context of Lie $H$-pseudoalgebras. We first introduce the notion of an $L_\infty$ $H$-pseudoalgebra (also called a strongly homotopy Lie $H$-pseudoalgebra or a sh Lie $H$-pseudoalgebra) which is the homotopy invariant extension of graded Lie $H$-pseudoalgebras. To define such a notion, we first consider a graded pseudotensor category $\mathcal{M}^*_{gr} (H)$ whose objects are graded left $H$-modules and (compositions of) polylinear maps are defined through the iterative use of the comultiplication map of $H$. With this, an $L_\infty$ $H$-pseudoalgebra is simply an $L_\infty$ algebra in the graded pseudotensor category $\mathcal{M}^*_{gr} (H)$. We give some equivalent descriptions of $L_\infty$ $H$-pseudoalgebras. Among others, given a graded left $H$-module $\mathcal{L}$, we construct a graded Lie algebra whose Maurer-Cartan elements correspond to $L_\infty$ $H$-pseudoalgebra structures on $\mathcal{L}$ (cf. Theorem \ref{mc-thm}). This construction generalizes the Maurer-Cartan characterization of Lie $H$-pseudoalgebras given in \cite{wu}. We also consider representations and cohomology of $L_\infty$ $H$-pseudoalgebras. Given an $L_\infty$ $H$pseudoalgebra $\mathcal{L}$, one can construct an $L_\infty$ algebra on the space $H^* \otimes_H \mathcal{L}$. This is called the annihilation algebra of $\mathcal{L}$. Next, let $\Gamma$ be a finite group acting on $H$ by Hopf algebra automorphisms. Let $\widetilde{H} = H \# {\bf k}[\Gamma]$ be the corresponding smash product. We observe that any $L_\infty$ $\widetilde{H}$-pseudoalgebra is equivalent to an $L_\infty$ $H$-pseudoalgebra equipped with an action of $\Gamma$ preserving the structure maps. 

We provide several examples of $L_\infty$ $H$-pseudoalgebras. First of all, we show that any $L_\infty$ conformal algebra introduced in \cite{sahoo-das} is an example of an $L_\infty$ $H$-pseudoalgebra. We also construct the current $L_\infty$ $H$-pseudoalgebra associated to an $L_\infty$ algebra. We define the notion of an $A_\infty$ $H$-pseudoalgebra and show that a suitable skew-symmetrization of an $A_\infty$ $H$-pseudoalgebra gives rise to an $L_\infty$ $H$-pseudoalgebra. 

Next, we focus on those $L_\infty$ $H$-pseudoalgebras whose underlying graded left $H$-module is concentrated in arity $0$ and $1$. We call them $2$-term $L_\infty$ $H$-pseudoalgebras. The collection of all $2$-term $L_\infty$ $H$-pseudoalgebras and morphisms between them forms a category, denoted by {\bf 2shLie}$H$. We also consider skeletal and strict $L_\infty$ $H$-pseudoalgebras and generalize some results of Baez and Crans \cite{baez-crans} in the present context. More precisely, we show that skeletal $L_\infty$ $H$-pseudoalgebras correspond to third cocycles of Lie $H$-pseudoalgebras and strict $L_\infty$ $H$-pseudoalgebras correspond to crossed modules of Lie $H$-pseudoalgebras.

\subsection{Categorification of Lie {\em H}-pseudoalgebras} In the final part of the paper, we consider the categorification of Lie $H$-pseudoalgebras. Namely, we introduce and study Lie-$2$ $H$-pseudoalgebras. When $H = {\bf k}$ is the base field, our notion coincides with the Lie-$2$ algebras considered by Baez and Crans. We show that the collection of all Lie-$2$ $H$-pseudoalgebras and morphisms between them forms a category, denoted by {\bf Lie2$H$}. Finally, we prove that the category {\bf Lie2$H$} is equivalent to the category {\bf 2shLie}$H$ (cf. Theorem \ref{last-thm}).

\medskip

The paper is organized as follows. In Section \ref{sec2}, we recall Lie $H$-pseudoalgebras and their cohomology theory. In Section \ref{sec3}, we introduce $L_\infty$ $H$-pseudoalgebras, their representations and cohomology. Various examples of $L_\infty$ $H$-pseudoalgebras are given in Section \ref{sec4}. We consider skeletal and strict $L_\infty$ $H$-pseudoalgebras in Section \ref{sec5} and provide their characterizations. Finally, the notion of Lie-$2$ $H$-pseudoalgebras is introduced in Section \ref{sec6} and proves Theorem \ref{last-thm}.

%\subsection{Organization of the paper}

\section{Lie {\em H}-pseudoalgebras and their cohomology}\label{sec2} In this section, we recall some necessary background on Lie $H$-pseudoalgebras that are required for the smooth reading of the present paper. More precisely, we recall Lie $H$-pseudoalgebras, their representations and cohomology theory. We will mainly follow the references \cite{bakalov-andrea-kac}, \cite{andrea-marchei}, \cite{wu}.

First, recall that the notion of a pseudotensor category was introduced by Beilinson and Drinfeld \cite{bd} to study Lie algebras, representations, cohomology etc. from categorical points of view. Lie $H$-pseudoalgebras are Lie algebras in a suitable pseudotensor category associated with a cocommutative Hopf algebra $H$.

\begin{definition}
   A {\em pseudotensor category} $\mathcal{C}$ is a class of objects together with the following data:

   (i) For any finite non-empty set $I$, a family of objects $\{ L_i \}_{i \in I}$ labelled by $I$ and an object $M$, there is a vector space $\mathrm{Lin} ( \{ L_i \}_{i \in I}, M ) $ of polylinear maps equipped with an action of the symmetric group $S_I$ on $\mathrm{Lin} ( \{ L_i \}_{i \in I}, M )$ by permuting the factors,

   (ii) for any finite non-empty sets $I, J$ and a surjective map $\pi : J \twoheadrightarrow I$, families of objects $\{ L_i \}_{i \in I}$, $\{ N_j \}_{j \in J}$ and an object $M$, there exists a composition map
   \begin{align}\label{comp-maps}
       \mathrm{Lin} ( \{ L_i \}_{i \in I}, M ) \otimes \bigotimes_{i \in I}  \mathrm{Lin} ( \{ N_j \}_{j \in \pi^{-1} (i)}, L_i ) \rightarrow  \mathrm{Lin} ( \{ N_j \}_{j \in J}, M), ~ \varphi \otimes ( \bigotimes_{i \in I} \psi_i) \mapsto \varphi ( \{ \psi_i \}_{i \in I})
   \end{align}
   subject to satisfy the following axioms:

   - (associativity) for another finite non-empty set $K$ and a surjective map $\theta: K \twoheadrightarrow J$, families of objects $\{ P_k \}_{k \in K}$ and given $\zeta_j \in \mathrm{Lin} ( \{ P_k \}_{k \in \theta^{-1} (j)}, N_j)$, we have
   \begin{align*}
       \varphi \big( \{  \psi_i  (  \{ \zeta_j \}_{j \in J} )      \}_{i \in I}   \big) = \big(    \varphi ( \{ \psi_i \}_{i \in I})      \big) (   \{ \zeta_j \}_{j \in J}),
   \end{align*}

   - (unit) for any object $M$, there exists an element $1_M \in \mathrm{Lin} ( \{ M \}, M)$ such that 
   \begin{align}\label{unit}
   1_M ( \varphi) = \varphi  (  \{1_{L_i} \}_{i \in I} ) = \varphi, \text{ for any } \varphi \in \mathrm{Lin} ( \{ L_i \}_{i \in I}, M),
\end{align}

   - (equivariance) the composition maps (\ref{comp-maps}) are equivariant with respect to the actions of the symmetric groups.
\end{definition}

\begin{definition}
(i) Let $\mathcal{C}$ be a pseudotensor category. A {\em Lie algebra} in the pseudotensor category $\mathcal{C}$ is an object $L$ equipped with a polylinear map $\beta \in \mathrm{Lin} (\{ L, L \}, L)$ that is skew-symmetric in the sense that $\beta = - \sigma_{12} \beta$ (where $\sigma_{12} = (1~2) \in S_2$) and satisfies the Jacobi identity
\begin{align}\label{jac-iden}
   \beta (\beta (\cdot, \cdot), \cdot) = \beta (\cdot , \beta (\cdot , \cdot)) - \sigma_{12} \beta (\cdot, \beta (\cdot, \cdot)).
\end{align}
We denote a Lie algebra as above simply by the pair $(L, \beta)$.

(ii) Let $(L, \beta)$ be a Lie algebra in a pseudotensor category $\mathcal{C}$. A {\em representation} of the Lie algebra $(L, \beta)$ is a pair $(M, \gamma)$, where $M$ is an object of $\mathcal{C}$ and $\gamma \in \mathrm{Lin} (\{ L, M \}, M)$ is a polylinear map satisfying
\begin{align*}
   \gamma (\beta (\cdot, \cdot), \cdot) = \gamma (\cdot , \gamma (\cdot , \cdot)) - \sigma_{12} \gamma (\cdot, \gamma (\cdot, \cdot)).
\end{align*}
It follows that any Lie algebra is a representation of itself, called the adjoint representation.
\end{definition}

Cohomology of a Lie algebra in a pseudotensor category $\mathcal{C}$ with coefficients in a representation is defined in \cite{bakalov-andrea-kac}. However, we will not require this general cohomology theory. At the end of this section, we revise the cohomology of a Lie $H$-pseudoalgebra which is sufficient for the present paper.

\begin{remark}
It is important to remark that any pseudotensor category $\mathcal{C}$ is naturally an ordinary category with the hom sets $ \mathrm{Hom}_\mathcal{C} (L, M):= \mathrm{Lin} (\{ L \}, M),$ for any two objects $ L$ and $ M$.
\end{remark}

\begin{exam}
Let $\mathcal{C} = \mathsf{Vect}$ be the category of all vector spaces over the field $\mathbf{k}$. Given any collection $\{ L_i \}_{i \in I}$ of vector spaces labelled by a non-empty finite set $I$ and another vector space $M$, we set $\mathrm{Lin} ( \{ L_i \}_{i \in I}, M) := \mathrm{Hom} ( \otimes_{i \in I} L_i, M)$. The symmetric group $S_I$ acts on $\mathrm{Lin} ( \{ L_i \}_{i \in I}, M)$ by permuting the factors in $\otimes_{i \in I} L_i$. This is an example of a pseudotensor category in which the composition map (\ref{comp-maps}) is defined by $\varphi (\{ \psi_i \}_{i \in I}) = \varphi \circ (\otimes_{i \in I} \psi_i)$.
\end{exam}
\begin{exam}
Let $H$ be a cocommutative bialgebra and $\mathcal{M}^l (H)$ be the category of all left $H$-modules. Given any collection $\{ L_i \}_{i \in I}$ of left $H$-modules and a left $H$-module $M$, we set $\mathrm{Lin} (\{ L_i \}_{i \in I}, M) := \mathrm{Hom}_H ( \otimes_{i \in I} L_i, M)$ on which the symmetric group $S_I$ acts by permuting the factors in $\otimes_{i \in I} L_i$. This is also a pseudotensor category where the composition map is defined similarly to the previous example.
\end{exam}

The most important example of a pseudotensor category for Lie $H$-pseudoalgebras is given by the following. Let $H$ be a cocommutative Hopf algebra with the comultiplication map $\Delta: H \rightarrow H^{\otimes 2}$. Given any finite non-empty set $I$, we define the tensor product functor $\mathcal{M}^l (H)^{\otimes I} \rightarrow \mathcal{M}^l (H^{\otimes I})$ by the notation $\boxtimes_{i \in I}$. Thus, if $\{L_i \}_{i \in I}$ is a collection of left $H$-modules then $\boxtimes_{i \in I} L_i$ is nothing but $\otimes_{i \in I} L_i$ with the left $H^{\otimes I}$-action given by $(\otimes_{i \in I} h_i) ( \otimes_{i \in I} x_i) = \otimes_{i \in I} h_i x_i$, for $\otimes_{i \in I} h_i \in H^{\otimes I}$ and $\otimes_{i \in I} x_i \in \otimes_{i \in I} L_i$. We consider the pseudotensor category $\mathcal{M}^* (H)$ with the same objects as $\mathcal{M}^l (H)$ (i.e. left $H$-modules) but with different polylinear maps (hence different composition maps),
\begin{align*}
    \mathrm{Lin} (\{ L_i \}_{i \in I}, M) := \mathrm{Hom}_{H^{\otimes I}} ( \boxtimes_{i \in I} L_i, H^{\otimes I} \otimes_H M). 
\end{align*}
Here the symmetric group $S_I$ acts on the space $\mathrm{Lin} (\{ L_i \}_{i \in I}, M)$ by simultaneously permuting the factors in $\boxtimes_{i \in I} L_i$ and $H^{\otimes I}$. This is well-defined as $H$ is cocommutative. For any surjective map $\pi : J \twoheadrightarrow I$ between finite non-empty sets, consider the functor $\Delta^{(\pi)} : \mathcal{M}^l (H^{\otimes I}) \rightarrow \mathcal{M}^l (H^{\otimes J})$, $M \mapsto H^{\otimes J} \otimes_{H^{\otimes I}} M$, where $H^{\otimes I}$ acts on $H^{\otimes J}$ via the iterative comultiplication determined by $\pi$. With these, the composition of polylinear maps in the pseudotensor category $\mathcal{M}^* (H)$ is given by $\varphi (\{ \psi_i \}_{i \in I}):= \Delta^{(\pi)} (\phi) \circ ( \boxtimes_{i \in I} \psi_i ).$

\begin{definition}
    A {\em Lie $H$-pseudoalgebra} (or a {\em Lie pseudoalgebra} if $H$ is clear from the context) is a Lie algebra $(L, \beta)$ in the pseudotensor category $\mathcal{M}^* (H)$.
\end{definition}

It follows from the above definition that a Lie $H$-pseudoalgebra is a left $H$-module $L$ equipped with a map $\beta \in \mathrm{Hom}_{H^{\otimes 2}} (L \boxtimes L, H^{\otimes 2} \otimes_H L)$, also called a {\em pseudobracket}, that is skew-symmetric and satisfies the Jacobi identity. We often write $\beta (x \otimes y)$ as $[x * y ]$, for $x, y \in L$. Note that the $H^{\otimes 2}$-linearity of the pseudobracket $\beta$ simply means that $[fx * gy] = ((f \otimes g) \otimes_H 1) [x * y]$, for all $x, y \in L$ and $f, g \in H$. Explicitly,
\begin{align*}
    \text{ if } [x * y] = \sum_{i} (f_i \otimes g_i) \otimes_H e_i ~ \text{ then } ~ [fx * gy] = \sum_{i} (ff_i \otimes gg_i) \otimes_H e_i.
\end{align*}
Similarly, the pseudobracket is skew-symmetric implies that $[y * x ] = -(\sigma_{12} \otimes_H 1) [x * y]$, for all $x , y \in L$. Explicitly,
\begin{align*}
    \text{ if } [x * y] = \sum_{i} (f_i \otimes g_i) \otimes_H e_i ~ \text{ then } ~ [y * x] = - \sum_{i} (g_i \otimes f_i) \otimes_H e_i.
    \end{align*}
Finally, to describe explicitly the Jacobi identity (\ref{jac-iden}) for the pseudobracket, we need to compute all the compositions (in $\mathcal{M}^* (H)$) that appeared in the Jacobi identity. Let $x, y, z \in L$. If 
\begin{align*}
[x * y] = \sum_{i} (f_i \otimes g_i) \otimes_H e_i ~~ \text{ and say } [e_i * z] = \sum_{j} (f_{ij} \otimes g_{ij}) \otimes_H e_{ij},
\end{align*}
then $[[x * y] * z]$ is the element of $H^{\otimes 3} \otimes_H L$ given by
\begin{align}\label{xy-z}
    [[x * y] * z] = \sum_{i, j} \big( f_i f_{ij (1)} \otimes g_i f_{ij (2)} \otimes g_{ij}   \big) \otimes_H e_{ij}.
\end{align}
Similarly, if $[y * z] = \sum_i (h_i \otimes l_i) \otimes_H d_i$ and say $[x * d_i] = \sum_j (h_{ij} \otimes l_{ij}) \otimes_H d_{ij}$, then
\begin{align}\label{x-yz}
    [x * [y * z]] = \sum_{i, j} \big( h_{ij} \otimes h_i l_{ij (1)} \otimes l_i l_{ij (2)} \big) \otimes_H d_{ij} ~ \text{ as an element of } H^{\otimes 3} \otimes_H L.
\end{align}
Similarly, we can define the composition $[y * [x * z]]$. Note the Jacobi identity (\ref{jac-iden}) for the pseudobracket is simply $[[x * y] * z] = [x * [y * z]] - [y * [x * z]]$, where all the terms are explicitly given above. We often denote a Lie $H$-pseudoalgebra simply by $L$ if the pseudobracket $\beta = [\cdot * \cdot]$ is clear from the context.

Let $(L, [\cdot * \cdot])$ and $(L', [\cdot * \cdot]')$ be two Lie $H$-pseudoalgebras. A {\em morphism} of Lie $H$-pseudoalgebras from $L$ to $L'$ is a $H$-linear map $\theta : L \rightarrow L'$ that satisfies $(\mathrm{id}_{H^{\otimes 2}} \otimes_H \theta) [x * y] = [\theta (x) * \theta(y)]'$, for all $x, y \in L.$

\medskip

Let $L$ be a Lie $H$-pseudoalgebra. A {\em representation} of $L$ is a left $H$-module $M$ equipped with a map $\gamma \in \mathrm{Hom}_{H^{\otimes 2}} (L \boxtimes M, H^{\otimes 2} \otimes_H M)$, called the {\em action map} written as $\gamma (x \otimes u) = x * u$, that satisfies
\begin{align}\label{rep-iden}
    [x * y] * u = x * (y * u) - y * ( x * u), \text{ for all } x , y \in L \text{ and } u \in M.
\end{align}
Here each composition in (\ref{rep-iden}) can be explicitly defined similarly to the above discussions.

Let $L$ be a Lie $H$-pseudoalgebra and $M$ be a representation. For each $n \geq 0$, we define an abelian group $C^n_{\text{Lie-}H} (L, M)$ by
\begin{align*}
    C^n_{\text{Lie-}H} (L, M) = \begin{cases}
        \mathbf{k} \otimes_H M & \text{ if } n=0,\\
        \{ \theta \in  \mathrm{Hom}_{H^{\otimes n}} (L^{\boxtimes n}, H^{\otimes n} \otimes_H M)~|~ \theta \text{ is skew-symmetric, i.e.}\\
        \qquad \qquad \qquad \qquad  \qquad \qquad \quad \qquad  \sigma \theta = (-1)^{\mathrm{sgn} (\sigma)} \theta, \forall \sigma \in S_n \} & \text{ if } n \geq 1.
    \end{cases}
\end{align*}
Then there is a map $\delta : C^n_{\text{Lie-}H} (L, M) \rightarrow C^{n+1}_{\text{Lie-}H} (L, M)$ given by
\begin{align*}
   ( \delta (1 \otimes_H u)) (x) = \sum_{i} \varepsilon (g_i) f_i \otimes_H u_i, ~~ \text{ if } x * u = \sum_i (f_i \otimes g_i) \otimes_H u_i
\end{align*}
and
\begin{align*}
    (\delta \theta) (x_1, \ldots, x_{n+1} ) =~& \sum_{i=1}^{n+1} (-1)^{i+1} (\sigma_{1 \rightarrow i} \otimes_H 1) ~ x_i * \theta (x_1, \ldots, \widehat{x_i} , \ldots, x_{n+1})\\
    +& \sum_{1 \leq i < j \leq n+1} (-1)^{i+j} (\sigma_{\substack{1 \rightarrow i \\
    2 \rightarrow j}} \otimes_H 1) ~ \theta ([x_i * x_j], x_1, \ldots, \widehat{x_i}, \ldots, \widehat{x_j}, \ldots, x_{n+1}),
\end{align*}
for $1 \otimes_H u \in C^0_{\text{Lie-}H} (L, M)$, $\theta \in C^{n \geq 1}_{\text{Lie-}H} (L, M)$ and $x, x_1, \ldots, x_{n+1} \in L$. Here $\sigma_{1 \rightarrow i}$ and $\sigma_{\substack{1 \rightarrow i \\
    2 \rightarrow j}}$ are respectively the permutations on $H^{\otimes n+1}$ given by
    \begin{align*}
       \sigma_{1 \rightarrow i}( h_i \otimes h_1 \otimes \cdots \otimes h_{i-1} \otimes h_{i+1} \otimes \cdots \otimes h_{n+1}) = h_1 \otimes \cdots \otimes h_{n+1},\\
       \sigma_{\substack{1 \rightarrow i \\
    2 \rightarrow j}} ( h_i \otimes h_j \otimes h_1 \otimes \widehat{h_i} \otimes \cdots \otimes \widehat{h_j} \otimes \cdots \otimes h_{n+1} ) = h_1 \otimes \cdots \otimes h_{n+1}.
    \end{align*}
    It has been observed in \cite{bakalov-andrea-kac} that $\delta^2 = 0$. The cohomology groups of the cochain complex $\{ C^\bullet_{\text{Lie-}H} (L, M), \delta \}$ are called the cohomology of $L$ with coefficients in the representation $M$. We denote the corresponding cohomology groups by $H^\bullet_{\text{Lie-}H} (L, M).$

\section{Strongly homotopy Lie {\em H}-pseudoalgebras}\label{sec3}
In this section, we first consider a (graded) pseudotensor category $\mathcal{M}^*_{gr} (H)$ generalizing $\mathcal{M}^*(H)$ in the graded context. A graded Lie algebra in the pseudotensor category $\mathcal{M}^*_{gr} (H)$ is called a graded Lie $H$-pseudoalgebra. Motivated by this, we define the notion of an $L_\infty$ $H$-pseudoalgebra as an $L_\infty$ algebra in the pseudotensor category $\mathcal{M}^*_{gr} (H)$. Finally, we give some equivalent descriptions of $L_\infty$ $H$-pseudoalgebras.

%Let $H$ be a cocommutative bialgebra and let $\mathcal{M}^l_{gr} (H)$ be the category of all graded left $H$-modules. For any collection

Let $\mathsf{grVect}$ be the category of all $\mathbb{Z}$-graded vector spaces. For any collection $\{ \mathcal{L}_i \}_{i \in I}$ of graded vector spaces and another graded vector space $\mathcal{M}$, we set $\mathrm{Lin} ( \{ \mathcal{L}_i \}_{i \in I}, \mathcal{M}) = \mathrm{Hom} (   \otimes_{i \in I} \mathcal{L}_i, \mathcal{M})$ the set of all graded linear maps from the graded vector space $ \otimes_{i \in I} \mathcal{L}_i$ to the graded vector space $\mathcal{M}$. Note that $\mathrm{Lin} ( \{ \mathcal{L}_i \}_{i \in I}, \mathcal{M})$ is a graded vector space with the summands 
\begin{align*}
    \mathrm{Lin} ( \{ \mathcal{L}_i \}_{i \in I}, \mathcal{M}) = \oplus_{n \in \mathbb{Z}} \mathrm{Lin}^n ( \{ \mathcal{L}_i \}_{i \in I}, \mathcal{M}) = \oplus_{n \in \mathbb{Z}} \mathrm{Hom}^n (   \otimes_{i \in I} \mathcal{L}_i , \mathcal{M}),
\end{align*}
where $\mathrm{Hom}^n (   \otimes_{i \in I} \mathcal{L}_i , \mathcal{M})$ is the space of all degree $n$ graded linear maps. The symmetric group $S_I$ also acts on the graded space $ \mathrm{Lin} ( \{ \mathcal{L}_i \}_{i \in I}, \mathcal{M})$ by $(\sigma f) (\otimes_{i \in I} x_i) = \epsilon (\sigma) f (\otimes_{i \in I} x_{\sigma^{-1} (i)})$, where $\epsilon (\sigma) = \epsilon (\sigma; \otimes_{i \in I} x_i)$ is the standard {\em Koszul sign} that appears in the graded context. The composition map $\varphi (\{ \psi_i \}_{i \in I})$ is given by
\begin{align}
    \big(  \varphi (\{ \psi_i \}_{i \in I}) \big) (\otimes_{j \in J} x_j) = \big(  \varphi \circ (\otimes_{i \in I} \psi_i)  \big) (\otimes_{j \in J} x_j) = \epsilon (  \{ \psi_i \}_{i \in I} ; \otimes_{j \in J} x_j) \varphi \big(  \otimes_{i \in I} \psi_i (  \otimes_{j \in \pi^{-1} (i)} x_j  )  \big).
\end{align}
Here $\epsilon (  \{ \psi_i \}_{i \in I} ; \otimes_{j \in J} x_j)$ is also the Koszul sign. The above composition maps turn out to be associative and equivariant with respect to the actions of the symmetric groups. Further, for any graded vector space $\mathcal{L}$, we have the identity map $\mathrm{id}_\mathcal{L} \in \mathrm{Lin}^0 (\{ \mathcal{L} \}, \mathcal{L})$ that satisfies (\ref{unit}).

Motivated by the above discussions, we introduce the following.

\begin{definition}
    A {\em graded pseudotensor category} $\mathcal{C}_{gr}$ (over $\mathbb{Z}$) is a class of $\mathbb{Z}$-graded objects together with graded vector spaces 
    \begin{align*}
        \mathrm{Lin} (\{ \mathcal{L}_i \}_{i \in I}, \mathcal{M}) = \oplus_{n \in \mathbb{Z}} \mathrm{Lin}^n (\{ \mathcal{L}_i \}_{i \in I}, \mathcal{M}),
    \end{align*}
    equipped with actions of the symmetric groups $S_I$ among them by permuting the factors and there are composition maps $\varphi ( \{ \psi_i \}_{i \in I})$ subject to satisfy the associativity, unit condition and the equivariance condition.
\end{definition}

Of course the category $\mathsf{grVect}$ is a graded pseudotensor category. More generally, let $H$ be a cocommutative Hopf algebra and $\mathcal{M}^l_{gr} (H)$ be the category of all graded left $H$-modules. For any collection $\{ \mathcal{L}_i \}$ of graded left $H$-modules and another graded left $H$-module $\mathcal{M}$, we set 
\begin{align*}
    \mathrm{Lin} ( \{ \mathcal{L}_i \}_{i \in I}, \mathcal{M}) = \oplus_{n \in \mathbb{Z}} \mathrm{Lin}^n ( \{ \mathcal{L}_i \}_{i \in I}, \mathcal{M}) = \oplus_{n \in \mathbb{Z}} \mathrm{Hom}^n_H (\otimes_{i \in I} \mathcal{L}_i, \mathcal{M}),
\end{align*}
where $ \mathrm{Hom}^n_H (\otimes_{i \in I} \mathcal{L}_i, \mathcal{M})$ is the space of all graded $H$-linear maps of degree $n$ from the graded space $\otimes_{i \in I} \mathcal{L}_i$ to $\mathcal{M}$. The symmetric group action and the compositions can be defined as the motivated example $\mathsf{grVect}$. With these, $\mathcal{M}^l_{gr} (H)$ turns out to be a graded pseudotensor category.

Next, we consider a graded pseudotensor category $\mathcal{M}^*_{gr} (H)$ as follows. Its objects are the same as the objects of $\mathcal{M}^l_{gr} (H)$ (i.e. graded left $H$-modules) and polylinear maps are defined by
\begin{align*}
    \mathrm{Lin} (\{ \mathcal{L}_i \}_{i \in I}, \mathcal{M}) = \oplus_{n \in \mathbb{Z}} \mathrm{Lin}^n (\{ \mathcal{L}_i \}_{i \in I}, \mathcal{M}) := \oplus_{n \in \mathbb{Z}} \mathrm{Hom}^n_{H^{\otimes I}} ( \boxtimes_{i \in I} \mathcal{L}_i. H^{\otimes I} \otimes_H \mathcal{M}). 
\end{align*}
The symmetric group actions and composition maps are graded analogues of the same defined for $\mathcal{M}^* (H)$.

\begin{definition}
    Let $\mathcal{C}_{gr}$ be a graded pseudotensor category (over $\mathbb{Z}$).

    (i) A {\em graded Lie algebra} in $\mathcal{C}_{gr}$ is a pair $(\mathcal{L}, \beta)$ consisting of an object $\mathcal{L}$ with a degree $0$ polylinear map $\beta \in \mathrm{Lin}^0 (\{ \mathcal{L}, \mathcal{L} \}, \mathcal{L})$ that is {graded} skew-symmetric (i.e. $\beta (\cdot , \cdot) = - \epsilon (\sigma_{12}) (\sigma_{12} \otimes_H 1) \sigma_{12} \beta (\cdot , \cdot)$) and satisfies the graded Jacobi identity 
    \begin{align*}
        \beta (\beta (\cdot , \cdot), \cdot) = \beta ( \cdot , \beta (\cdot , \cdot)) - \epsilon (\sigma_{12}) \sigma_{12} \beta (\cdot , \beta (\cdot, \cdot)).
    \end{align*}

    (ii) A {\em differential graded Lie algebra} in $\mathcal{C}_{gr}$ is a graded Lie algebra $(\mathcal{L}, \beta)$ equipped with a degree $-1$ polylinear map $d \in \mathrm{Lin}^{-1} (\{ \mathcal{L} \}, \mathcal{L})$ satisfying
    \begin{align*}
        {d}^2 = 0 ~~~~ \text{ and } ~~~~ d \beta (\cdot, \cdot) = \beta (d \cdot, \cdot) + \epsilon (\beta; \cdot, \cdot) \beta (\cdot, d \cdot).
    \end{align*}
    \end{definition}

    \begin{definition}
        A {\em graded Lie $H$-pseudoalgebra} is a graded Lie algebra in the graded pseudotensor category $\mathcal{M}^*_{gr} (\mathcal{H})$. Similarly, a {\em differential graded Lie $H$-pseudoalgebra} is a differential graded Lie algebra in the graded pseudotensor category $\mathcal{M}^*_{gr} (\mathcal{H})$.
    \end{definition}

    Thus, a graded Lie $H$-pseudoalgebra is a pair $(\mathcal{L}, \beta)$ consisting of a graded left $H$-module $\mathcal{L} = \oplus_{n \in \mathbb{Z}} \mathcal{L}^n$ with a degree $0$ graded pseudobracket $\beta \in \mathrm{Hom}^0_{H^{\otimes 2}} (\mathcal{L} \boxtimes \mathcal{L}, H^{\otimes 2} \otimes_H \mathcal{L})$, written as $\beta (x \otimes y) = [x * y]$, satisfying 

    - (graded skew-symmetry) $[y * x ] = -(-1)^{|x| |y|} (\sigma_{12} \otimes_H 1) [x * y]$,

    - (graded Jacobi identity) for any homogeneous elements $x, y, z \in \mathcal{L}$,
    \begin{align*}
        [[x * y] * z] = [x * [y * z]] - (-1)^{|x| |y|} [y * [x * z]].
    \end{align*}
Further, it is a differential graded Lie $H$-pseudoalgebra if there exists a $H$-linear degree $-1$ map $d : \mathcal{L} \rightarrow H \otimes_H \mathcal{L}$ satisfies $d^2 = 0$ and $d [x * y] = [dx * y] + (-1)^{|x|} [x * dy]$, for all $x, y \in \mathcal{L}$.

    In the following, we will introduce $L_\infty$ $H$-pseudoalgebras. Before we give the explicit description, we motivate the reader by giving the following.

    \begin{definition}
        Let $\mathcal{C}_{gr}$ be a graded pseudotensor category (over $\mathbb{Z}$). An {\em $L_\infty$ algebra} in $\mathcal{C}_{gr}$ is a pair $(\mathcal{L}, \{ \beta_k \}_{k \geq 1})$ consisting of an object $\mathcal{L} = \oplus_{n \in \mathbb{Z}} \mathcal{L}^n$ together with a collection
        \begin{align*}
    \big\{ \beta_k \in \mathrm{Lin}^{k-2} ( \{ \underbrace{\mathcal{L}, \ldots, \mathcal{L}}_{k  \mathrm{ ~many}} \}, \mathcal{L} ) \big\}_{k \geq 1}
    \end{align*}
    of polylinear maps with specific degrees (namely, $\mathrm{deg} (\beta_k) = k-2$, for $k \geq 1$) that satisfy the following axioms:

    - ((graded) skew-symmetry) each $\beta_k$ is skew-symmetric in the sense that $\beta_k = (-1)^\sigma (\sigma \beta_k)$, for all $\sigma \in S_k$,

    - (higher Jacobi identities) for each $N \in \mathbb{N}$ and homogeneous elements $x_1, \ldots, x_N \in \mathcal{L}$, the following identity holds:
    \begin{align}\label{higher-j}
        \sum_{k+l = N+1} \sum_{\sigma \in \mathrm{Sh} (l, k-1)} (-1)^\sigma \epsilon (\sigma) (-1)^{l (k-1)} \beta_k \big( \beta_l (x_{\sigma (1)}, \ldots, x_{\sigma (l)}), x_{\sigma (l+1)}, \ldots, x_{\sigma (N)} \big) = 0.
    \end{align}
    \end{definition}

    Let $\mathcal{C}_{gr} = \mathsf{grVect}$ be the graded pseudotensor category of all graded vector spaces. It follows from the above definition that an $L_\infty$ algebra in $\mathsf{grVect}$ is a pair $(\mathcal{L}, \{ \beta_k \}_{k \geq 1})$ consisting of a graded vector space $\mathcal{L} = \oplus_{n \in \mathbb{Z}} \mathcal{L}^n$ with a collection $\{ \beta_k \in \mathrm{Hom}^{k-2} (\mathcal{L}^{\otimes k}, \mathcal{L}) \}_{k \geq 1}$ of graded skew-symmetric maps satisfying the higher Jacobi identities (\ref{higher-j}). This is a classical $L_\infty$ algebra considered in \cite{lada-s,lada-markl}. Thus, an $L_\infty$ algebra in a graded pseudotensor category generalizes classical $L_\infty$ algebras. Motivated by this observation, we now define the following.

\begin{definition}\label{shlieH}
    An {\em $L_\infty$ $H$-pseudoalgebra} (or simply a {\em strongly homotopy Lie $H$-pseudoalgebra}) is an $L_\infty$ algebra in the graded pseudotensor category $\mathcal{M}^*_{gr} (H)$. More precisely, an $L_\infty$ $H$-pseudoalgebra is a pair $(\mathcal{L}, \{ \beta_k \}_{k \geq 1})$ consisting of a graded left $H$-module $\mathcal{L} = \oplus_{n \in \mathbb{Z}} \mathcal{L}^n$ equipped with a collection $\big\{ \beta_k \in \mathrm{Hom}_{H^{\otimes k}} (  \mathcal{L}^{\boxtimes k}, H^{\otimes k} \otimes_H \mathcal{L}  )  \big\}_{k \geq 1}$ of polylinear maps with $\mathrm{deg} (\beta_k) = k-2$, for $k \geq 1$, subject to the following conditions:

    - (graded skew-symmetry) each $\beta_k$ is graded skew-symmetric in the sense that
    \begin{align*}
        \beta_k (x_1, \ldots, x_k) = (-1)^\sigma \epsilon (\sigma) (\sigma \otimes_H 1) \beta_k (  x_{\sigma^{-1} (1)}, \ldots, x_{\sigma^{-1} (k)}), \text{ for all } \sigma \in S_k,
    \end{align*}

    - (higher Jacobi identities) For each $N \in \mathbb{N}$ and homogeneous elements $x_1, \ldots, x_N \in \mathcal{L}$, the identity (\ref{higher-j}) holds.

\end{definition}

\begin{remark}
    (i) Any graded Lie $H$-pseudoalgebra $(\mathcal{L}, \beta)$ can be realized as an $L_\infty$ $H$-pseudoalgebra $(\mathcal{L}, \{ \beta_k \}_{k \geq 1})$ with $\beta_2 = \beta$ and $\beta_k = 0$ for all $k \neq 2$. Similarly, a differential graded Lie $H$-pseudoalgebra $(\mathcal{L}, \beta, d) $ is an $L_\infty$ $H$-pseudoalgebra with $\beta_1 = d$, $\beta_2 = \beta$ and $\beta_k = 0$ for all $k \neq 1, 2$.

    (ii) As expected, any classical $L_\infty$ algebra $(\mathcal{L}, \{ \beta_k \}_{k \geq 1})$ is an $L_\infty$ $H$-pseudoalgebra for $H = \mathbf{k}$.
\end{remark}

There is another equivalent description of an $L_\infty$ $H$-pseudoalgebra that renders some other simple descriptions. Let $(\mathcal{L}, \{ \beta_k \}_{k \geq 1})$ be an $L_\infty$ $H$-pseudoalgebra. Consider the graded left $H$-module $\mathcal{L}[-1]$ given by $\mathcal{L}[-1] = \oplus_{n \in \mathbb{Z}} (\mathcal{L}[-1])^n = \oplus_{n \in \mathbb{Z}} \mathcal{L}^{n-1}$. For any $k \geq 1$, we define
\begin{align}\label{eta}
     \eta_k \in \mathrm{Hom}_{H^{\otimes k}} \big( (\mathcal{L}[-1])^{\boxtimes k}, H^{\otimes k} \otimes_H \mathcal{L}[-1]   \big) ~~\text{ by } ~~ \eta_k = (-1)^{\frac{k (k-1)}{2}} s \circ \beta_k \circ (s^{-1})^{\otimes k},
\end{align}
where $s: \mathcal{L} \rightarrow \mathcal{L} [-1]$ is the degree $+1$ map that identifies $\mathcal{L}$ and $\mathcal{L}[-1]$, and $s^{-1} : \mathcal{L} [-1] \rightarrow \mathcal{L}$ is the inverse of $s$.  Since $\mathrm{deg}(\beta_k) = k-2$, it follows that $\mathrm{deg} (\eta_k )= -1$.  The graded skew-symmetry of $\beta_k$ is equivalent to the fact that the map $\eta_k$ is graded symmetric. Furthermore, the higher Jacobi identities (\ref{higher-j}) are equivalent to the identities
\begin{align}\label{shifted-j}
    \sum_{k+l = N+1} \sum_{\sigma \in \mathrm{Sh} (l, k-1)} \epsilon (\sigma) \eta_k \big( \eta_l ( u_{\sigma (1)}, \ldots, u_{\sigma (l)}), u_{\sigma (l+1)}, \ldots, u_{\sigma (N)} \big) = 0,
\end{align}
for all $N \in \mathbb{N}$ and homogeneous elements $u_1, \ldots, u_N \in \mathcal{L}[-1]$. As a summary, we obtain the following.

\begin{proposition}\label{prop-shifted}
    Let $\mathcal{L} = \oplus_{n \in \mathbb{Z}} \mathcal{L}^n$ be a graded left $H$-module. An $L_\infty$ $H$-pseudoalgebra structure on $\mathcal{L}$ is equivalent to a collection of degree $-1$ polylinear maps $\{  \eta_k \in \mathrm{Hom}^{-1}_{H^{\otimes k}} (   (\mathcal{L} [-1])^{\otimes k}, H^{\otimes k} \otimes_H \mathcal{L}[-1]) \}_{k \geq 1}$ that are graded symmetric and satisfying the shifted higher Jacobi identities (\ref{shifted-j}).
\end{proposition}

Next, we will consider another characterization of $L_\infty$ $H$-pseudoalgebras in terms of Maurer-Cartan elements in a suitable graded Lie algebra. Let $\mathcal{W} = \oplus_{n \in \mathbb{Z}} \mathcal{W}^n$ be a graded left $H$-module. For any $p \in \mathbb{Z}$ and $k \geq 1$, let $\mathrm{Hom}^p_{H^{\otimes k}} (\mathcal{W}^{\boxtimes k}, H^{\otimes k} \otimes_H \mathcal{W})$ be the set of all graded polylinear maps in $\mathrm{Hom}_{H^{\otimes k}} (\mathcal{W}^{\boxtimes k}, H^{\otimes k} \otimes_H \mathcal{W})$ of degree $p$.  We also let $\mathrm{symHom}^p_{H^{\otimes k}} (\mathcal{W}^{\boxtimes k}, H^{\otimes k} \otimes_H \mathcal{W})$ be the set of all maps in  $\mathrm{Hom}^p_{H^{\otimes k}} (\mathcal{W}^{\boxtimes k}, H^{\otimes k} \otimes_H \mathcal{W})$ that are graded symmetric. We define
\begin{align*}
    \mathrm{Hom}^p_{\mathrm{poly}} (\mathcal{W}, \mathcal{W}) :=~& \oplus_{k \geq 1} \mathrm{Hom}^p_{H^{\otimes k}} (\mathcal{W}^{\boxtimes k}, H^{\otimes k} \otimes_H \mathcal{W}),\\
     \mathrm{symHom}^p_{\mathrm{poly}} (\mathcal{W}, \mathcal{W}) :=~& \oplus_{k \geq 1} \mathrm{symHom}^p_{H^{\otimes k}} (\mathcal{W}^{\boxtimes k}, H^{\otimes k} \otimes_H \mathcal{W}).
\end{align*}
Thus, an element $\eta \in \mathrm{Hom}^p_\mathrm{poly} (\mathcal{W}, \mathcal{W})$ (resp. $\mathrm{Hom}^p_\mathrm{poly} (\mathcal{W}, \mathcal{W})$) is a formal sum $\eta = \sum_{k \geq 1} \eta_k$, where $\eta_k \in \mathrm{Hom}^p_{H^{\otimes k}} ( \mathcal{W}^{\boxtimes k}, H^{\otimes k} \otimes_H \mathcal{W}) $  (resp. $\mathrm{symHom}^p_{H^{\otimes k}} ( \mathcal{W}^{\boxtimes k}, H^{\otimes k} \otimes_H \mathcal{W})$) for $k \geq 1$. It has been implicitly observed by Wu \cite{wu} that the graded space $\oplus_{p \in \mathbb{Z}} \mathrm{symHom}^p_\mathrm{poly} (\mathcal{W}, \mathcal{W})$ carries a graded Lie bracket given by
\begin{align}\label{gl-bracket}
    \llbracket \sum_{k \geq 1} \eta_k , \sum_{l \geq 1} \zeta_l   \rrbracket = \sum_{N \geq 1} \sum_{k+l = N+1} \llbracket \eta_k, \zeta_l \rrbracket := \sum_{N \geq 1} \sum_{k+l = N+1} \big(  \eta_k \diamond \zeta_l - (-1)^{pq} \zeta_l \diamond \eta_k  \big), 
\end{align}
where
\begin{align}\label{compo}
    (\eta_k \diamond \zeta_l) (w_1, \ldots, w_N) := \sum_{\sigma \in \mathrm{Sh} (l, k-1)} \epsilon (\sigma) \eta_k \big( \zeta_l ( w_{\sigma (1)}, \ldots, w_{\sigma (l)}), w_{\sigma (l+1)}, \ldots, w_{\sigma (N)} \big),
\end{align}
for $\eta = \sum_{k \geq 1} \eta_k \in \mathrm{symHom}^p_\mathrm{poly} (\mathcal{W}, \mathcal{W})$,  $\zeta = \sum_{l \geq 1} \zeta_l \in \mathrm{symHom}^q_\mathrm{poly} (\mathcal{W}, \mathcal{W})$ and $w_1, \ldots, w_N \in \mathcal{W}$. 
%The graded Lie bracket considered in (\ref{gl-bracket}) is a generalization of the classical Nijenhuis-Richardson bracket.

\begin{thm}\label{mc-thm}
    Let $\mathcal{L} = \oplus_{n \in \mathbb{Z}} \mathcal{L}^n$ be a graded left $H$-module. Then there is a one-to-one correspondence between $L_\infty$ $H$-pseudoalgebra structures on $\mathcal{L}$ and Maurer-Cartan elements in the graded Lie algebra $\big( \oplus_{p \geq 1} \mathrm{symHom}^p_{\mathrm{poly}} (\mathcal{W}, \mathcal{W}), \llbracket ~, ~ \rrbracket   \big)$, where $\mathcal{W} = \mathcal{L}[-1]$.
\end{thm}

\begin{proof}
    Let $\{  \beta_k \in \mathrm{Hom}_{H^{\otimes k}}^{k-2} (\mathcal{L}^{\boxtimes k}, H^{\otimes k} \otimes_H \mathcal{L})    \}_{k \geq 1}$ be a collection of graded skew-symmetric polylinear maps of specific degrees. Then we have seen in (\ref{eta}) that the above collection is equivalent to having a collection $\{  \eta_k \in \mathrm{Hom}_{H^{\otimes k}}^{-1} (\mathcal{W}^{\boxtimes k}, H^{\otimes k} \otimes_H \mathcal{W}) \}_{k \geq 1}$ of graded symmetric polylinear maps. If we consider the element $\eta= \sum_{k \geq 1} \eta_k \in \mathrm{symHom}_\mathrm{poly}^{-1} (\mathcal{W}, \mathcal{W})$, then we have from (\ref{gl-bracket}) that
    \begin{align*}
        \llbracket \eta, \eta \rrbracket = \sum_{N \geq 1} \sum_{k+l = N+1} (\eta_k \diamond \eta_l + \eta_l \diamond \eta_k) = 2 \sum_{N \geq 1} \sum_{k+l = N+1} (\eta_k \diamond \eta_l),
    \end{align*}
    where $(\eta_k \diamond \eta_l)$ is given in (\ref{compo}). This show that $\eta = \sum_{k \geq 1} \eta_k$ is a Maurer-Cartan element in the graded Lie algebra $(\oplus_{p \geq 1} \mathrm{symHom}^p_{\mathrm{poly}} (\mathcal{W}, \mathcal{W}), \llbracket ~, ~ \rrbracket )$ if and only if the collection $\{ \eta_k \}_{k \geq 1}$ satisfy the identities (\ref{shifted-j}). Hence the result follows after using Proposition \ref{prop-shifted}.
\end{proof}

\begin{remark}
    The above theorem generalizes the Maurer-Cartan characterization of Lie $H$-pseudoalgebras considered by Wu \cite{wu}.
\end{remark}

As a consequence of all the above discussions, we get the following equivalent descriptions of an $L_\infty$ $H$-pseudoalgebra:

(i) $(\mathcal{L} , \{ \beta_k \}_{k \geq 1})$ is an $L_\infty$ $H$-pseudoalgebra,

(ii) there exists a collection of degree $-1$ polylinear maps $\{ \eta_k \in \mathrm{Hom}_{H^{\otimes k}}  (  (\mathcal{L}[-1])^{\otimes k}, H^{\otimes k} \otimes_H \mathcal{L}[-1]) \}_{k \geq 1}$ that are graded symmetric and satisfy the shifted higher Jacobi identities (\ref{shifted-j}),

(iii) there exists a degree $-1$ square-zero coderivation $\eta$ on the coassociative coalgebra $(T^c_H (\mathcal{L}[-1]), \Delta),$

(iv) there exists a Maurer-Cartan element in the graded Lie algebra $\big( \oplus_{p \geq 1} \mathrm{Hom}^p_{\mathrm{poly}} (\mathcal{L}[-1], \mathcal{L}[-1]), \llbracket ~, ~ \rrbracket   \big)$.

\medskip

\subsection{Representations and cohomology of strongly homotopy Lie $H$-pseudoalgebras}
In this subsection, we first consider representations of an $L_\infty$ $H$-pseudoalgebra. Then using the Maurer-Cartan characterization of $L_\infty$ $H$-pseudoalgebras, we define their cohomology (with coefficients in representations).

Let $(\mathcal{L} = \oplus_{n \in \mathbb{Z} } \mathcal{L}^n , \{ \beta_k \}_{k \geq 1} )$ be an $L_\infty$ $H$-pseudoalgebra. A {\em representation} of the $L_\infty$ $H$-pseudoalgebra is given by a pair $( \mathcal{M} = \oplus_{n \in \mathbb{Z} } \mathcal{M}^n, \{ \gamma_k \}_{k \geq 1} )$ consisting of a graded left $H$-module $\mathcal{M} = \oplus_{n \in \mathbb{Z}} \mathcal{M}^n$ with a collection of polylinear maps $\big\{  \gamma_k \in \mathrm{Hom}_{H^{\otimes k}} (\mathcal{L}^{\boxtimes k-1} \boxtimes \mathcal{M}, H^{\otimes k} \otimes_H \mathcal{M}) \big\}_{k \geq 1}$ with $\mathrm{deg} (\gamma_k) = k-2$ for $k \geq 1$, that subject to satisfy the following conditions: 

(i) each $\gamma_k$ is skew-symmetric on simultaneous permutations of the factors of $\mathcal{L}^{\boxtimes k-1}$ and first $k-1$ factors of $H^{\otimes k}$,

(ii) for each $N \in \mathbb{N}$, the identity (\ref{higher-j}) holds when $x_1, \ldots, x_{N-1} \in \mathcal{L}$ and $x_N \in \mathcal{M}$ (with the convension that $\beta_k = \gamma_k$ when $\sigma (N) =N$ and $\beta_l = \gamma_l$ when $\sigma (l) = N$).

It follows from the above definition that an $L_\infty$ $H$-pseudoalgebra $(\mathcal{L}, \{ \beta_k \}_{k \geq 1} )$ is a representation of itself. This is called the {\em adjoint representation}.

\begin{remark}
    Any representation of a Lie $H$-pseudoalgebra (see Equation (\ref{rep-iden})) can be seen as a representation of the corresponding $L_\infty$ $H$-pseudoalgebra concentrated in degree $0$.
\end{remark}

Given a Lie $H$-pseudoalgebra and a representation of it, one can construct the semidirect product Lie $H$-pseudoalgebra \cite{bakalov-andrea-kac}. This can be generalized in the strongly homotopy context.

\begin{proposition}\label{prop-semid}
    Let $(\mathcal{L}, \{ \beta_k \}_{k \geq 1} )$ be an $L_\infty$ $H$-pseudoalgebra and $(\mathcal{M} , \{ \gamma_k \}_{k \geq 1})$ be a representation of it. Then $(\mathcal{L} \oplus \mathcal{M}, \{ \overline{\beta}_k \}_{k \geq 1} )$ is an $L_\infty$ $H$-pseudoalgebra, where the structure maps $\{ \overline{\beta}_k \}_{k \geq 1}$ are given by 
    \begin{align}\label{semi-br}
        (\overline{\beta}_k) \big(  (x_1, u_1), \ldots, (x_k, u_k) \big) := \big(  \beta_k (x_1, \ldots, x_k) , \sum_{i=1}^k (-1)^{k-i} ~\gamma_k (x_1, \ldots, \widehat{x_i}, \ldots, x_k, u_i) \big).
    \end{align}
\end{proposition}

\medskip

Let $(\mathcal{L}, \{ \beta_k \}_{k \geq 1} )$ be an $L_\infty$ $H$-pseudoalgebra and $\mathcal{M}$ be a graded left $H$-module equipped with a collection of polylinear maps $\{  \gamma_k \in \mathrm{Hom}^{k-2}_{H^{\otimes k}} (\mathcal{L}^{\boxtimes k-1} \boxtimes \mathcal{M}, H^{\otimes k} \otimes_H \mathcal{M}) \}_{k \geq 1}$ such that each $\gamma_k$ is skew-symmetric on simultaneous permutations of the factors of $\mathcal{L}^{\boxtimes k-1}$ and first $k-1$ factors of $H^{\otimes k}$. Then on the direct sum graded left $H$-module $\mathcal{L} \oplus \mathcal{M}$, we define a collection $\{ \overline{\beta}_k \}_{k \geq 1}$ of skew-symmetric polylinear maps by (\ref{semi-br}). Then it can be easily checked that $(\mathcal{L} \oplus \mathcal{M}, \{ \overline{\beta}_k \}_{k \geq 1})$ is an $L_\infty$ $H$-pseudoalgebra if and only if $(\mathcal{M} , \{ \gamma_k \}_{k \geq 1})$ is a representation of the $L_\infty$ $H$-pseudoalgebra $(\mathcal{L}, \{ \beta_k \}_{k \geq 1} )$.

Let $(\mathcal{L}, \{ \beta_k \}_{k \geq 1} )$ be an $L_\infty$ $H$-pseudoalgebra and $(\mathcal{M} , \{ \gamma_k \}_{k \geq 1})$ be a representation of it. Consider the semidirect product $L_\infty$ $H$-pseudoalgebra $(\mathcal{L} \oplus \mathcal{M} , \{ \overline{\beta}_k \}_{k \geq 1})$ given in Proposition \ref{prop-semid}. This $L_\infty$ $H$-pseudoalgebra gives rise to the Maurer-Cartan element $\overline{ \eta} = \sum_{k \geq 1} \overline{\eta}_k$ in the graded Lie algebra
\begin{align*}
    \big( \oplus_{p \in \mathbb{Z}} \mathrm{Hom}^p_\mathrm{poly} (   (\mathcal{L} \oplus \mathcal{M}) [-1],   (\mathcal{L} \oplus \mathcal{M}) [-1]  ), \llbracket ~, ~ \rrbracket   \big),
\end{align*}
where $\overline{\eta}_k = (-1)^{\frac{k (k-1)}{2}} s \circ \overline{\beta}_k \circ (s^{-1})^{\otimes k}$ for $k \geq 1$. This yields a cochain complex $\{ C^\bullet_{\text{shLie-}H} (\mathcal{L} \oplus \mathcal{M}, \mathcal{L} \oplus \mathcal{M}), \delta_{\overline{\eta}} \}$ induced by the Maurer-Cartan element $\overline{\eta}$. More precisely,
\begin{align*}
    C^n_{\text{shLie-}H} (\mathcal{L} \oplus \mathcal{M}, & \mathcal{L} \oplus \mathcal{M}) := \mathrm{Hom}_{\text{poly}}^{-(n-1)} \big( (\mathcal{L} \oplus \mathcal{M})[-1], (\mathcal{L} \oplus \mathcal{M})[-1]   \big) ~~~ \text{ and } \\
   & \delta_{\overline{\eta}} (\eta) := (-1)^{n-1} \llbracket \overline{\eta} , \eta \rrbracket, \text{ for } \eta \in  C^n_{\text{shLie-}H}.
\end{align*}
For each $n \in \mathbb{Z}$, we define a subspace  $C^n_{\text{shLie-}H} (\mathcal{L}, \mathcal{M}) \subset  C^n_{\text{shLie-}H} (\mathcal{L} \oplus \mathcal{M}, \mathcal{L} \oplus \mathcal{M})$ by
\begin{align*}
     C^n_{\text{shLie-}H} (\mathcal{L} , \mathcal{M}) := \mathrm{Hom}_{\text{poly}}^{-(n-1)} \big( \mathcal{L} [-1], \mathcal{M}[-1]   \big) = \oplus_{k \geq 1} \mathrm{Hom}^{-(n-1)}_{H^{\otimes k}}  \big(  ( \mathcal{L} [-1])^{\otimes k}, H^{\otimes k} \otimes_H \mathcal{M}[-1]  \big).
\end{align*}
Then it can be easily seen that $\delta_{\overline{\eta}}  \big( C^n_{\text{shLie-}H} (\mathcal{L} , \mathcal{M})   \big) \subset C^{n+1}_{\text{shLie-}H} (\mathcal{L} , \mathcal{M})$. In other words, we get a subcomplex $\{ C^\bullet_{\text{shLie-}H} (\mathcal{L}, \mathcal{M}), \delta_{\overline{\eta}} \}$ of the cochain complex $\{ C^\bullet_{\text{shLie-}H} (\mathcal{L} \oplus \mathcal{M},  \mathcal{L} \oplus \mathcal{M}), \delta_{\overline{\eta}} \}$. The cohomology groups of this subcomplex are called the cohomology groups of the $L_\infty$ $H$-pseudoalgebra $\mathcal{L}$ with coefficients in the representation $\mathcal{M}$.    

\subsection{$\Gamma$-actions} Let $H$ be a cocommutative Hopf algebra and $\Gamma$ be a group (not necessarily finite) acting on $H$ by Hopf algebra automorphisms. We denote the action of $\Gamma$ on $H$ by $g \cdot f$, for $g \in \Gamma$ and $f \in H$. That is, $g \cdot f = g f g^{-1}$. Consider the smash product $\widetilde{H} = H \# {\bf k}[\Gamma]$. As an associative algebra, it is the semi-direct product of $H$ and ${\bf k}[\Gamma]$, and as a coassociative coalgebra it is the tensor product of coalgebras. Note that, for any finite nonempty set $I$, there is a map $\delta_I : \widetilde{H}^{\otimes I} \rightarrow H^{\otimes I} \otimes_H \widetilde{H}$ given by
\begin{align*}
    \delta_I (\otimes_{i \in I} f_i g_i) = \big(  \otimes_{i \in I} f_i  \big) \otimes g ~~ \text{ if all } g_i\text{'s are equal to some } g
\end{align*}
and zero otherwise. Let $\mathcal{L} = \oplus_{n \in \mathbb{Z}} \mathcal{L}_n$ be a graded left $\widetilde{H}$-module (thus, an object in $\mathcal{M}^*_{gr} (   \widetilde {H})$). This is equivalent to the fact that $\mathcal{L}$ is a graded left $H$-module equipped with an action of $\Gamma$ on it that is compatible with the action of $H$. More precisely, the correspondence is given by
\begin{align*}
    (g \cdot f)(x) = g (f (g^{-1} x)) , \text{ for } f \in H, g \in \Gamma \text{ and } x \in \mathcal{L}.
\end{align*}
Let $\mathcal{M}^*_{gr, \Gamma} (H)$ be the graded subcategory of $\mathcal{M}^*_{gr} (H)$ whose objects are graded left $\widetilde{H}$-modules and with those graded polylinear maps of $\mathcal{M}^*_{gr} (H)$ that commute with the action of $\Gamma$ in the sense that
\begin{align*}
    \varphi ( \{ g \psi_i \}_{i \in I}) = g \cdot \varphi (\{ \psi_i \}_{i \in I}), \text{ for all } g \in \Gamma.
\end{align*}
Then $\mathcal{M}^*_{gr, \Gamma} (H)$ is also a graded pseudotensor category. Moreover, there is a graded pseudotensor functor $\delta : \mathcal{M}^*_{gr} (\widetilde{H}) \rightarrow \mathcal{M}^*_{gr, \Gamma} (H)$ defined as follows. For an object $\mathcal{L}$ in $\mathcal{M}^*_{gr} (\widetilde{H})$ (i.e. a graded left $\widetilde{H}$-module), we set $\delta (\mathcal{L}) = \mathcal{L}$. For a polylinear map $\varphi \in \mathrm{Lin} (\{ \mathcal{L}_i \}_{i \in I}, \mathcal{M})$ in $\mathcal{M}^*_{gr} (\widetilde{H})$, we set $\delta (\varphi)$ to be the composition
\begin{align*}
    \boxtimes_{i \in I} \mathcal{L}_i \xrightarrow{ \varphi} \widetilde{H}^{\otimes I} \otimes_{ \widetilde{H}} \mathcal{M} \xrightarrow{ \delta_I \otimes_{\widetilde{H}} \mathrm{id} } (H^{\otimes I} \otimes_H \widetilde{H}) \otimes_{\widetilde{H}} \mathcal{M} \xrightarrow{ \cong } H^{\otimes I} \otimes_H \mathcal{M}.
\end{align*}
Since the maps $\delta_I$ are compatible with the actions of the symmetric groups and with the comultiplication of $\widetilde{H}$, it follows that $\delta$ is compatible with the actions of the symmetric groups and with compositions of polylinear maps.

The proof of the following result is similar to \cite[Theorem 5.1]{bakalov-andrea-kac}.

\begin{proposition}
    For a finite group $G$, the functor $\delta : \mathcal{M}^*_{gr} (\widetilde{H}) \rightarrow \mathcal{M}^*_{gr, \Gamma} (H)$ is an equivalence of graded pseudotensor categories.
\end{proposition}

As a corollary of the above result, we get the following.

\begin{thm}
    Let $H$ be a cocommutative Hopf algebra and $\Gamma$ be a finite group. Let $\widetilde{H} = H  \# {\bf k}[\Gamma]$. Then an $L_\infty$ $\widetilde{H}$-pseudoalgebra is same as an $L_\infty$ $H$-pseudoalgebra $(\mathcal{L}, \{ \beta_k \}_{k \geq 1})$ equipped with an action of $\Gamma$ on the graded left $H$-module $\mathcal{L}$ that satisfies
    \begin{align}\label{invar}
        \beta_k (gx_1, \ldots, gx_k) = g \cdot \beta_k (x_1, \ldots, x_k ),
    \end{align}
    for $k \geq 1$; $x_1, \ldots, x_k \in \mathcal{L}$ and $g \in \Gamma$.
\end{thm}

Let $\{ \beta_k \}_{k \geq 1}$ be an $L_\infty$ $H$-pseudoalgebra structure on $\mathcal{L}$ that satisfies (\ref{invar}). Then the $L_\infty$ $\widetilde{H}$-pseudoalgebra structure $\{  \widetilde{\beta}_k \}_{k \geq 1}$ on $\mathcal{L}$ is given by
\begin{align*}
    \widetilde{\beta}_k (x_1, \ldots, x_k) = \sum_{g \in \Gamma} ((g^{-1} \otimes 1 \otimes \cdots \otimes 1) \otimes_{\widetilde{H}} 1) \beta_k (gx_1, \ldots, x_k ), \text{ for } k \geq 1.
\end{align*}

\subsection{Annihilation algebras}
Let $H$ be a cocommutative Hopf algebra with the multiplication $\Delta$ and count $\varepsilon$. Let $Y$ be an $H$-bimodule which is also a commutative $H$-differential algebra for both the left and right actions of $H$.

For any graded left $H$-module $\mathcal{L} = \oplus_{n \in \mathbb{Z}} \mathcal{L}^n$, we define $\mathcal{A}_Y \mathcal{L} := Y \otimes_H \mathcal{L}$. Then there is a left $H$-action on $\mathcal{A}_Y \mathcal{L}$ given by $h (y \otimes_H x) = hy \otimes_H x$, for $h \in H$, $y \in Y$ and $x \in \mathcal{L}$. Additionally, if $(\mathcal{L}, \{ \beta_k \}_{k \geq 1})$ is an $L_\infty$ $H$-pseudoalgebra, then we can define a collection of products $\{ \widetilde{\beta}_k : (\mathcal{A}_Y \mathcal{L})^{\otimes k} \rightarrow \mathcal{A}_Y \mathcal{L} \}_{k \geq 1}$ by
\begin{align*}
    \widetilde{\beta}_k \big(  &y_1 \otimes_H x_1, \ldots, y_k \otimes_H x_k  \big) = \sum_i (y_1 f_{1 i}) \cdots (y_k f_{k i}) \otimes_H e_i\\
    &  ~~~~ \text{ if } \beta_k (x_1, \ldots, x_k) = \sum_i (f_{1 i} \otimes \cdots \otimes f_{k i}) \otimes_H e_i.
\end{align*}
It is easy to see that these maps are well-defined. Moreover, we have the following result.
\begin{proposition}
    Let $(\mathcal{L}, \{ \beta_k \}_{k \geq 1})$ be an $L_\infty$ $H$-pseudoalgebra. Then the pair $(\mathcal{A}_Y \mathcal{L}, \{ \widetilde{\beta}_k \}_{k \geq 1})$ is an $L_\infty$ algebra whose structure maps additionally satisfy
    \begin{align*}
        h \big( \widetilde{\beta}_k (   y_1 \otimes_H x_1, \ldots, y_k \otimes_H x_k  ) \big) = \widetilde{\beta}_k \big( h_{(1)} (y_1 \otimes_H x_1), \ldots, h_{(k)} (y_k \otimes_H x_k) \big), \text{ for any } h \in H.
    \end{align*}
\end{proposition}

\begin{proof}
Let $y_1 \otimes_H x_1, \ldots, y_N \otimes_H x_N \in \mathcal{A}_Y \mathcal{L}$ and $\sigma \in \mathrm{Sh}(l, k-1)$ be any shuffle with $k+l = N+1$. If
\begin{align*}
    \beta_k \big( \beta_l (x_{\sigma (1)}, \ldots, x_{\sigma (l)}), x_{\sigma (l+1)}, \ldots, x_{\sigma (N)} \big) = \sum_i (f^\sigma_{1 i} \otimes \cdots \otimes f^\sigma_{N i}) \otimes_H e_i
\end{align*}
then we observe that
\begin{align*}
&\widetilde{\beta}_k \big( \widetilde{\beta}_l (   y_{\sigma(1)} \otimes_H  x_{\sigma (1)}, \ldots,  y_{\sigma(l)} \otimes_H x_{\sigma (l)}),  y_{\sigma(l+1)} \otimes_H x_{\sigma (l+1)}, \ldots,  y_{\sigma(N)} \otimes_H x_{\sigma (N)}    \big) \\
&= \sum_i (y_{\sigma(1)} f^\sigma_{1i}) \cdots (y_{\sigma(N)} f^\sigma_{Ni}) \otimes_H e_i \\
&= (y_{\sigma(1)} \cdots y_{\sigma (N)}) (\mu_{H^{\otimes N}} \otimes 1) \beta_k \big( \beta_l (x_{\sigma (1)}, \ldots, x_{\sigma (l)}), x_{\sigma (l+1)}, \ldots, x_{\sigma (N)} \big).
\end{align*}
Here $\mu_{H^{\otimes N}} : H^{\otimes N} \rightarrow H$ is the multiplication of $H$. Hence the result follows after writing the expression of the higher Jacobi identities. The last part follows from the definition.
\end{proof}

Note that the base field ${\bf k}$ is an $H$-bimodule and a commutative associative $H$-differential algebra with the action $h 1 = \varepsilon (h)$, for $h \in H$ and $1 \in {\bf k}$. Then it follows from the above Proposition that if $(\mathcal{L}, \{ \beta_k \}_{k \geq 1})$ is an $L_\infty$ $H$-pseudoalgebra, then $(\mathcal{A}_{\bf k} \mathcal{L} = {\bf k} \otimes_H \mathcal{L}, \{ \widetilde{\beta}_k \}_{k \geq 1})$ is an $L_\infty$ algebra. Explicitly, we have $\mathcal{A}_{\bf k} \mathcal{L} = \mathcal{L}/ H_{+} \mathcal{L}$, where $H_{+} \mathcal{L} = \{ h \cdot x ~|~ h \in H \text{ with } \varepsilon (h) = 0 \text{ and } x \in \mathcal{L} \}$ and the operations $\{ \widetilde{\beta}_k \}_{k \geq 1}$ are given by
\begin{align*}
    \widetilde{\beta}_k \big( [x_1], \ldots, [x_k] \big) = \sum_i [\varepsilon (f_{1i}) \cdots \varepsilon (f_{ki}) e_i] \quad 
    \text{ if } \beta_k (x_1, \ldots, x_k) = \sum_i (f_{1i} \otimes \cdots \otimes f_{ki}) \otimes_H e_i.
\end{align*}
Here $[x]$ denotes the class of the element $x \in \mathcal{L}$ in the quotient $\mathcal{L}/H_{+} \mathcal{L}$.

\begin{definition}
    Let $Y = H^*$ be the dual of $H$. Then $Y = H^*$ is an $H$-bimodule and cocommutative associative $H$-differential algebra. For any $L_\infty$ $H$-pseudoalgebra $\mathcal{L} = (\mathcal{L} , \{ \beta_k \}_{k \geq 1})$, the $L_\infty$ algebra 
    \begin{align*}
        \mathcal{A} (\mathcal{L}) := (\mathcal{A}_{H^*} (\mathcal{L}) = H^* \otimes_H \mathcal{L}, \{ \widetilde{\beta}_k \}_{k \geq 1})
    \end{align*}
    is called the {\em annihilation algebra} of $\mathcal{L}$.
\end{definition}

In a forthcoming article, we will discuss annihilation algebra in more detail. More precisely, we aim to study differential $L_\infty$ algebras (following the work of Ritt \cite{ritt} about differential Lie algebras) and find their relations with annihilation algebras.

 \section{Examples of strongly homotopy Lie {\em H}-pseudoalgebras}\label{sec4}

 \subsection{$L_\infty$ conformal algebras} The concept of Lie conformal algebras was introduced by Kac \cite{kac,kac2} to study the operator product expansions of chiral fields in conformal field theory. More precisely, a {\em Lie conformal algebra} is a ${\bf k}[\partial]$-module $L$ equipped with a ${\bf k}$-linear map $[\cdot_\lambda \cdot] : L \otimes L \rightarrow L[\lambda]$, called the {\em $\lambda$-bracket} that satisfy

 - (conformal sesquilinearity) $[\partial x_\lambda y] = - \lambda [ x _\lambda y]$ and $[x_\lambda \partial y] = (\partial + \lambda) [x_\lambda y]$,

 - (skew-symmetry) $[y_\lambda x] = - [x_{-\lambda - \partial} y],$

 - (conformal Jacobi identity) $[x_\lambda [y_\nu z]] = [[x_\lambda y]_{\lambda + \nu} z] + [y_\nu [x_\lambda z]].$

 \medskip

\noindent It has been observed in \cite{bakalov-andrea-kac} that Lie conformal algebras are exactly Lie $H$-pseudoalgebras, where $H = {\bf k}[\partial]$ is the Hopf algebra of polynomials in one variable $\partial$. The correspondence between the $\lambda$-bracket and the pseudobracket is given by
 \begin{align*}
     [x_\lambda y] = \sum_i p_i (\lambda) e_i ~~~~ \text{ if and only if } ~~~~ [x * y] = \sum_i \big( p_i (-\partial) \otimes 1 \big) \otimes_{   {\bf k}[\partial]  } e_i.
 \end{align*}

 On the other hand, $L_\infty$ conformal algebras were introduced in \cite{sahoo-das} as the strongly homotopy analogue of Lie conformal algebras. Recall that an {\em $L_\infty$ conformal algebra} is a graded ${\bf k}[\partial]$-module $\mathcal{L} = \oplus_{n \in \mathbb{Z}} \mathcal{L}^n$ equipped with a collection
 \begin{align*}
     \{ l_k : \mathcal{L}^{\otimes k} \rightarrow \mathcal{L}[\lambda_1, \ldots, \lambda_{k-1}], ~ x_1 \otimes \cdots \otimes x_k \mapsto (l_k)_{\lambda_1, \ldots, \lambda_{k-1}} (x_1, \ldots, x_k) \text{ with } \text{deg} (l_k) = k-2 \}_{k \geq 1}
 \end{align*}
 of ${\bf k}$-linear maps that satisfy the following conditions:

 - each $l_k$ is conformal sesquilinear in the sense that
 \begin{align*}
     (l_k)_{\lambda_1, \ldots, \lambda_{k-1}} (x_1, \ldots, \partial x_i, \ldots ,x_{k-1}, x_k) =~& - \lambda_i (l_k)_{\lambda_1, \ldots, \lambda_{k-1}} (x_1, \ldots, x_k),\\
     (l_k)_{\lambda_1, \ldots, \lambda_{k-1}} (x_1, \ldots,x_{k-1}, \partial x_k) =~& (\partial + \lambda_1 + \cdots + \lambda_k) (l_k)_{\lambda_1, \ldots, \lambda_{k-1}} (x_1, \ldots, x_k),
 \end{align*}

 - each $l_k$ is graded skew-symmetric in the sense that
 \begin{align*}
 (l_k)_{\lambda_{1},\ldots,\lambda_{k-1}}(x_{1},\ldots,x_{k}) = \mathrm{sgn} (\sigma) \epsilon (\sigma)~ (l_k)_{\lambda_{\sigma (1)},\ldots,\lambda_{\sigma(k-1)}}(x_{\sigma(1)},\ldots,x_{\sigma (k-1)},x_{\sigma (k)}) \big|_{\lambda_k \mapsto \lambda_k^\dagger}, \text{ for } \sigma \in {S}_k,
 \end{align*}

 - for each $N \geq 1$ and homogeneous $x_1, \ldots, x_N \in \mathcal{L}$, the following higher conformal Jacobi identity holds:
\begin{align}
            \sum_{p+q = N+1} & \sum_{\sigma \in \mathrm{Sh}(q,N-q)}\mathrm{sgn}(\sigma) \epsilon(\sigma)(-1)^{q(p-1)} \\
           & (l_{p})_{  \lambda^\dagger_\sigma ,  \lambda_{\sigma (q+1)}, \ldots, \lambda_{\sigma (N-1)}  } \big(  (l_{q})_{ \lambda_{\sigma (1)}, \ldots, \lambda_{\sigma (q-1)}} (x_{\sigma(1)},\ldots,x_{\sigma(q)}),x_{\sigma(q+1)},\ldots,x_{\sigma(N)} \big) = 0. \nonumber
        \end{align}
        Note that $\sigma \in \mathrm{Sh}(q, N-q)$ implies that either $\sigma (q) = N$ or $\sigma (N) = N$. When $\sigma (q) = N$, we use the notation $\lambda_\sigma^\dagger =  - \partial - (\lambda_{\sigma (q+1)} + \cdots + \lambda_{\sigma (N)})$ and when $\sigma (N) = N$, we use the notation $\lambda_\sigma^\dagger = \lambda_{\sigma (1)} + \cdots + \lambda_{\sigma (q)}$ in the above identities.

 Then it can be checked that an $L_\infty$ conformal algebra is the same as an $L_\infty$ {\bf k}$[\partial]$-pseudoalgebra. More precisely,  $(\mathcal{L}, \{ l_k \}_{k \geq 1})$ is an $L_\infty$ conformal algebra if and only if $(\mathcal{L}, \{ \beta_k \}_{k \geq 1})$ is an  $L_\infty$ {\bf k}$[\partial]$-pseudoalgebra, where the correspondence between the structure maps are given by
 \begin{align*}
     (l_k)_{\lambda_1, \ldots, \lambda_{k-1}} (x_1, \ldots, x_k) = \sum_i p_i^{k, 1} (\lambda_1) \cdots p_i^{k, k-1} (\lambda_{k-1}) e_i ~~~~ \text{ if and only if } \\
     \beta_k (x_1, \ldots, x_k) = \big( p_i^{k, 1} (-\partial) \otimes \cdots \otimes p_i^{k, k-1} (-\partial) \otimes 1 \big) \otimes_{  {\bf k}[\partial] } e_i.
 \end{align*}

 \subsection{Current $L_\infty$ $H$-pseudoalgebras} Given a Lie algebra $\mathfrak{g}$, one can define a Lie $H$-pseudoalgebra structure on the left $H$-module $H \otimes \mathfrak{g}$ with the pseudobracket
 \begin{align*}
     [(f \otimes x) * (g \otimes y)] := ( f \otimes g) \otimes_H (1 \otimes [x, y]_\mathfrak{g}), \text{ for } f \otimes x, g \otimes y \in H \otimes \mathfrak{g}.
 \end{align*}
 This is called the {\em current Lie $H$-pseudoalgebra} associated to $\mathfrak{g}$, and it is denoted by $\mathrm{Cur}(\mathfrak{g})$. See \cite{bakalov-andrea-kac} for more details. This construction can be generalized to the homotopy context. More precisely, let $(\mathcal{L}, \{ l_k \}_{k \geq 1})$ be an $L_\infty$ algebra. Then the graded left $H$-module $H \otimes \mathcal{L}$ equipped with the graded polylinear maps 
 \begin{align*}
 \big\{ \beta_k \in \mathrm{Hom}_{H^{\otimes k}}^{k-2} \big(  (H \otimes \mathcal{L})^{\boxtimes k}, H^{\otimes k} \otimes_H (H \otimes \mathcal{L})   \big) \big\}_{k \geq 1}
 \end{align*}
 forms a $L_\infty$ $H$-pseudoalgebra, where
 \begin{align*}
     \beta_k (f_1 \otimes x_1, \ldots, f_k \otimes x_k) = (f_1 \otimes \cdots \otimes f_k) \otimes_H (1 \otimes l_k (x_1, \ldots, x_k)), \text{ for any } k \geq 1.
 \end{align*}
 The verification is straightforward.
 This is called the current $L_\infty$ $H$-pseudoalgebra associated to the $L_\infty$ algebra $(\mathcal{L}, \{ l_k \}_{k \geq 1})$. We denote it by $\mathrm{Cur}(\mathcal{L}).$

 More generally, let $H' \subset H$ be a Hopf subalgebra and $(\mathcal{L}, \{ \beta_k' \}_{k \geq 1})$ be an $L_\infty$ $H'$-pseudoalgebra. Then it can be verified that the pair $\mathrm{Cur}_{H'}^H \mathcal{L} = (H \otimes_{H'} \mathcal{L}, \{ \beta_k \}_{k \geq 1})$ is an $L_\infty$ $H$-pseudoalgebra, where the structure maps are given by
 \begin{align*}
     \beta_k \big( f_1 \otimes_{H'} x_1, \ldots, f_k \otimes_{H'} x_k \big) = \sum_i \big(  f_1 f_{1, i} \otimes \cdots \otimes f_k f_{k, i} \big) \otimes_H (1 \otimes_{H'} e_i) \\
     \text{ whenever } \beta_k' (x_1, \ldots, x_k) = \sum_{i} (f_{1, i} \otimes \cdots \otimes f_{k, i}) \otimes_{H'} e_i.
 \end{align*}

 \subsection{Graded Lie $H$-pseudoalgebras whose each arity is a free $H$-module of rank $1$} Let $\mathcal{L} = \oplus_{n \in \mathbb{Z}} H e_n$ be a graded left $H$-module whose each arity is a free $H$-module of rank $1$. Note that any polylinear map $\beta_k \in \mathrm{Hom}^{k-2}_{H^{\otimes k}} (\mathcal{L}^{\boxtimes k}, H^{\otimes k} \otimes_H \mathcal{L})$ is determined by the values $\beta_k (e_{i_1}, e_{i_2}, \ldots, e_{i_k})$, or equivalently, by elements $\alpha_{i_1, \ldots, i_k} \in H^{\otimes k}$ such that $\beta_k (e_{i_1}, \ldots, e_{i_k}) = \alpha_{i_1, \ldots, i_k} \otimes_H e_{i_1 + \cdots + i_k + k-2}$.

 Then we have the following.

 \begin{proposition}
     Let $\mathcal{L} = \oplus_{n \in \mathbb{Z}} H e_n$ be a graded left $H$-module whose each arity is a free $H$-module of rank $1$. Suppose there are maps $\{ \beta_k \in \mathrm{Hom}^{k-2}_{H^{\otimes k}} (\mathcal{L}^{\boxtimes k}, H^{\otimes k} \otimes_H \mathcal{L}) \}_{k \geq 1}$ which are determined by the elements $\alpha_{i_1, \ldots, i_k}$ ($k \geq 1$ and $i_1, \ldots, i_k \in \mathbb{Z}$). Then $(\mathcal{L}, \{ \beta_k \}_{k \geq 1})$ is a $L_\infty$ $H$-pseudoalgebra if and only if
     \begin{align*}
         \alpha_{i_{\sigma (1)}, \ldots, i_{\sigma (k)}} & = (-1)^\sigma \epsilon (\sigma; i_1, \ldots, i_k) \alpha_{i_1, \ldots, i_k}, \text{ for all } k \geq 1 \text{ and } \sigma \in S(k),\\
         \sum_{k+l = N+1} \sum_{\sigma \in \mathrm{Sh}(l, k-1)} &(-1)^\sigma \epsilon(\sigma) (-1)^{l(k-1)} \\  & \big( \alpha_{i_{\sigma (1)}, \ldots, i_{\sigma (l)}} \otimes 1^{\otimes k-1} \big) (\Delta^l \otimes \mathrm{id}^{\otimes k-1}) \alpha_{i_{\sigma (1)} + \cdots + i_{\sigma (l)} + l-2, i_{\sigma (l+1)}, \ldots, i_{\sigma (N)}} = 0,
     \end{align*}
     for all $N \in \mathbb{N}$. Here $\Delta^l : H \rightarrow H^{\otimes l}$ is the iterative comultiplication map.
 \end{proposition}

 \begin{proof}
     For any $\sigma \in \mathrm{Sh} (l, k-1)$, note that
     \begin{align*}
         \beta_l (x_{\sigma (1)}, \ldots, x_{\sigma (l)}) = \alpha_{ i_{\sigma (1)}, \ldots, i_{\sigma (l)} } \otimes_H e_{ i_{\sigma (1)} + \cdots + i_{\sigma (l)} + l-2}.
     \end{align*}
     Therefore,
     \begin{align*}
        & \beta_k (     \beta_l ( x_{\sigma (1)}, \ldots, x_{\sigma (l)}), x_{\sigma(l+1)}, \ldots,  x_{\sigma (N)}  ) \\
        & = \big( \alpha_{i_{\sigma (1)}, \ldots, i_{\sigma (l)}} \otimes 1^{\otimes k-1} \big) (\Delta^l \otimes \mathrm{id}^{\otimes k-1}) \alpha_{i_{\sigma (1)} + \cdots + i_{\sigma (l)} + l-2, i_{\sigma (l+1)}, \ldots, i_{\sigma (N)}}.
     \end{align*}
     Hence the result follows by using the definition of an $L_\infty$ $H$-pseudoalgebra.
 \end{proof}

 \begin{remark}
     Let $\mathcal{L} = \oplus_{n \in \mathbb{Z}} H e_n$ be a graded left $H$-module whose each arity is a free $H$-module of rank $1$. Suppose there is a map $\beta \in  \mathrm{Hom}^{0}_{H^{\otimes 2}} (\mathcal{L}^{\boxtimes 2}, H^{\otimes 2} \otimes_H \mathcal{L})$ which is determined by the elements $\alpha_{i, j} \in H^{\otimes 2}$ (for $i, j \in \mathbb{Z}$). Then $(\mathcal{L} , \beta)$ is a graded Lie $H$-pseudoalgebra if and only if 
     \begin{align*}
         \alpha_{i, j} =~& - (-1)^{ij} \alpha_{j, i}, \\
         (\alpha_{i, j} \otimes 1) (\Delta \otimes \mathrm{id}) \alpha_{i+j,k } =~& (1 \otimes \alpha_{j, k}) (\mathrm{id} \otimes \Delta) \alpha_{i, j+k} - (-1)^{ij} (\sigma_{12} \otimes \mathrm{id}) \big(  (1 \otimes \alpha_{i, k}) (\mathrm{id} \otimes \Delta) \alpha_{j, i+k}    \big) = 0.
     \end{align*}
 \end{remark}

\subsection{Strongly homotopy associative $H$-pseudoalgebras} In this subsection, we first introduce the notion of $A_\infty$ $H$-pseudoalgebras as the strongly homotopy version of associative $H$-pseudoalgebras. In the end, we show that a suitable skew-symmetrization of an $A_\infty$ $H$-pseudoalgebra gives rise to an $L_\infty$ $H$-pseudoalgebra.

\begin{definition} (\cite{bakalov-andrea-kac}) Let $\mathcal{C}$ be a pseudotensor category. An {\em associative algebra} in $\mathcal{C}$ is a pair $(A, \mu)$ of an object $A$ with a polylinear map $\mu \in \mathrm{Lin} (\{ A, A \}, A)$ satisfying the associativity $\mu (\mu (\cdot , \cdot), \cdot) = \mu (\cdot, \mu (\cdot, \cdot))$.

Let $H$ be a cocommutative Hopf algebra (although the cocommutativity is not required to define associative $H$-pseudoalgebras). An {\em associative $H$-pseudoalgebra} is an associative algebra in the pseudotensor category $\mathcal{M}^* (H)$. It follows that an associative $H$-pseudoalgebra is a left $H$-module $A$ with a map $\mu \in \mathrm{Hom}_{H^{\otimes 2}} (A \boxtimes A, H^{\otimes 2} \otimes_H A)$, often called the {\em pseudoproduct} denoted by $\mu (a \otimes b) = a * b$, satisfying $(a * b) * c = a * (b * c)$, for all $a, b, c \in A$. Here the compositions are similar to the one given in (\ref{xy-z}) and (\ref{x-yz}).

One can also define graded associative algebras, differential graded associative algebras and $A_\infty$ algebras in a graded pseudotensor category. If we consider the graded pseudotensor category $\mathcal{M}_{gr}^* (H)$ associative to a cocommutative Hopf algebra $H$, an $A_\infty$ algebra in $\mathcal{M}_{gr}^* (H)$ is simply called an $A_\infty$ $H$-pseudoalgebra. More precisely, we have the following.

\begin{definition}
    An {\em $A_\infty$ $H$-pseudoalgebra} is a pair $(\mathcal{A}, \{ \mu_k \}_{k \geq 1})$ consisting of a graded left $H$-module $\mathcal{A} = \oplus_{n \in \mathbb{Z}} \mathcal{A}^n$ equipped with a collection of polylinear maps
    \begin{align*}
        \{ \mu_k \in \mathrm{Hom}_{H^{\otimes k}}^{k-2} ( \mathcal{A}^{\boxtimes k}, H^{\otimes k} \otimes_H \mathcal{A}) \}_{k \geq 1}
    \end{align*}
    subject to satisfy the following higher associativity identities: for each $N \geq 1$ and homogeneous elements $a_1, \ldots, a_N \in \mathcal{A}$,
    \begin{align}\label{higher-asso}
        \sum_{k+l = N+1} \sum_{\lambda = 1}^k (-1)^{  \lambda (l+1) + l (   |a_1| + \cdots + |a_{\lambda -1}|)} \mu_k \big( a_1, \ldots, a_{\lambda -1} , \mu_l (a_\lambda, \ldots, a_{\lambda + l-1}), a_{\lambda + l}, \ldots, a_N \big) = 0.
    \end{align}
\end{definition}

As before, $\mu_k$ is $H^{\otimes k}$-linear simply means that
$\mu_k (h_1 a_1 , \ldots, h_k a_k) = (( h_1 \otimes \cdots \otimes h_k) \otimes_H 1) \mu_k (a_1, \ldots, a_k)$, for all $a_1, \ldots, a_k \in \mathcal{A}$ and $h_1, \ldots, h_k \in H$. Finally, to explicitly describe the higher associativity identities, we will write down explicitly each term under the summation of (\ref{higher-asso}). First, for any fixed $k, l \geq 1$ with $k+l = N+1$, suppose
\begin{align*}
    \mu_l (a_\lambda, \ldots, a_{\lambda + l -1}) = \sum_{i}  (h^1_i \otimes \cdots \otimes h^l_i) \otimes_H e_i \in H^{\otimes l} \otimes_H \mathcal{A}
\end{align*}
and
\begin{align*}
    \mu_k (a_1, \ldots, a_{\lambda -1}, e_i, a_{\lambda + 1}, \ldots, a_N) = \sum_{j} ( h^{\lambda, 1}_{ij} \otimes \cdots \otimes h^{\lambda, k}_{ij}) \otimes_H e_{ij} \in H^{\otimes k} \otimes_H \mathcal{A}.
\end{align*}
Then $\mu_k \big( a_1, \ldots, a_{\lambda -1} , \mu_l (a_\lambda, \ldots, a_{\lambda + l-1}), a_{\lambda + l}, \ldots, a_N \big)$ is the element of $H^{\otimes N} \otimes_H \mathcal{A}$ given by
\begin{align*}
    \sum_{i, j} \big(  h^{\lambda, 1}_{ij} \otimes \cdots \otimes h^{\lambda, \lambda -1}_{ij} \otimes h^1_i h^{\lambda, \lambda}_{ij (1)} \otimes \cdots \otimes h^l_i h^{\lambda, \lambda}_{ij (l)} \otimes h^{\lambda, \lambda+1}_{ij} \otimes \cdots \otimes h^{\lambda, k}_{ij}   \big) \otimes_H e_{ij}.
\end{align*}
\end{definition}

\begin{remark}
   (i) When $H = \mathbf{k}$, an $A_\infty$ $H$-pseudoalgebra is nothing but a classical $A_\infty$ algebra \cite{stas}.

    (ii) Any associative $H$-pseudoalgebra can be regarded as an $A_\infty$ $H$-pseudoalgebra concentrated in arity $0$.
\end{remark}

Let $\mathcal{W} = \oplus_{n \in \mathbb{Z}} \mathcal{W}^n$ be a graded left $H$-module. 
%For any $p \in \mathbb{Z}$ and $k \geq 1$, let $\mathrm{Hom}^p_{H^{\otimes k}} (\mathcal{W}^{\boxtimes k}, H^{\otimes k} \otimes_H \mathcal{W})$ be the set of all graded polylinear maps in $\mathrm{Hom}_{H^{\otimes k}} (\mathcal{W}^{\boxtimes k}, H^{\otimes k} \otimes_H \mathcal{W})$ of degree $p$. We define
%\begin{align*}
%    \mathrm{Hom}^p_{\mathrm{poly}} (\mathcal{W}, \mathcal{W}) := \oplus_{k \geq 1} \mathrm{Hom}^p_{H^{\otimes k}} (\mathcal{W}^{\boxtimes k}, H^{\otimes k} \otimes_H \mathcal{W}).
%\end{align*}
%Thus, an element $\eta \in \mathrm{Hom}^p_\mathrm{poly} (\mathcal{W}, \mathcal{W})$ is a formal sum $\eta = \sum_{k \geq 1} \eta_k$, where $\eta_k \in \mathrm{Hom}^p_{H^{\otimes k}} ( \mathcal{W}^{\boxtimes k}, H^{\otimes k} \otimes_H \mathcal{W}) $ for $k \geq 1$. It has been implicitly observed by Wu \cite{wu} that 
Then the graded space $\oplus_{p \in \mathbb{Z}} \mathrm{Hom}^p_\mathrm{poly} (\mathcal{W}, \mathcal{W})$ carries a graded Lie bracket given by
\begin{align*}
    \llbracket \sum_{k \geq 1} \nu_k , \sum_{l \geq 1} \theta_l   \rrbracket^{\sim} = \sum_{N \geq 1} \sum_{k+l = N+1} \llbracket \nu_k, \theta_l \rrbracket^{\sim} := \sum_{N \geq 1} \sum_{k+l = N+1} \big(  \nu_k \odot \theta_l - (-1)^{pq} \theta_l \odot \nu_k  \big), 
\end{align*}
where
\begin{align*}
    (\eta_k \odot \zeta_l) (w_1, \ldots, w_N) := \sum_{i =1}^k (-1)^{ |w_1| + \cdots + |w_{i-1}|}~ \nu_k \big( w_1, \ldots, w_{i-1}, \theta_l ( w_{i}, \ldots, w_{i+l-1}), w_{i+l}, \ldots, w_{N} \big),
\end{align*}
for $\sum_{k \geq 1} \nu_k \in \mathrm{Hom}^p_\mathrm{poly} (\mathcal{W}, \mathcal{W})$,  $\sum_{l \geq 1} \theta_l \in \mathrm{Hom}^q_\mathrm{poly} (\mathcal{W}, \mathcal{W})$ and $w_1, \ldots, w_N \in \mathcal{W}$. Then similar to Theorem \ref{mc-thm}, here one can prove the following result.

\begin{thm}
    Let $\mathcal{A}= \oplus_{n \in \mathbb{Z}} \mathcal{A}^n$ be a graded left $H$-module. Then there is a one-to-one correspondence between $A_\infty$ $H$-pseudoalgebra structures on $\mathcal{A}$ and Maurer-Cartan elements in the graded Lie algebra $(\oplus_{p \in \mathbb{Z}} \mathrm{Hom}^p_\mathrm{poly} (\mathcal{W}, \mathcal{W}), \llbracket ~, ~ \rrbracket^\sim  )$, where $\mathcal{W} = \mathcal{A} [-1]$.
\end{thm}

\begin{remark}
Note that the graded Lie algebra $(\oplus_{p \in \mathbb{Z}} \mathrm{Hom}^p_\mathrm{poly} (\mathcal{W}, \mathcal{W}), \llbracket ~, ~ \rrbracket^\sim  )$ constructed above has a close connection with the graded Lie algebra $(\oplus_{p \in \mathbb{Z}} \mathrm{symHom}^p_\mathrm{poly} (\mathcal{W}, \mathcal{W}), \llbracket ~, ~ \rrbracket  )$ defined in Section \ref{sec3}. More precisely, for any $p \in \mathbb{Z}$, the map $( - )^{\mathrm{sym}} : \mathrm{Hom}^p_\mathrm{poly} (\mathcal{W}, \mathcal{W}) \rightarrow \mathrm{symHom}^p_\mathrm{poly} (\mathcal{W}, \mathcal{W})$ defined by 
\begin{align*}
(\sum_{k \geq 1} \nu_k)^{\mathrm{sym}} =~& \sum_{k \geq 1} (\nu_k)^{\mathrm{sym}}, \text{ where} \\
     (\nu_k)^{\mathrm{sym}} (w_1, \ldots, w_k) :=~& \sum_{\sigma \in S(k)} \epsilon (\sigma) (\sigma \otimes_H 1) \nu_k (w_{\sigma (1)}, \ldots, w_{\sigma (k)})
\end{align*}
defines a morphism from $(\oplus_{p \in \mathbb{Z}} \mathrm{Hom}^p_\mathrm{poly} (\mathcal{W}, \mathcal{W}), \llbracket ~, ~ \rrbracket^\sim  )$ to $(\oplus_{p \in \mathbb{Z}} \mathrm{symHom}^p_\mathrm{poly} (\mathcal{W}, \mathcal{W}), \llbracket ~, ~ \rrbracket  )$.
\end{remark}

In \cite{bakalov-andrea-kac} the authors showed that an associative $H$-pseudoalgebra gives rise to a Lie $H$-pseudoalgebra if we take the skew-symmetrization of the associative pseudoproduct. More precisely, let $(A, \mu)$ be an associative $H$-pseudoalgebra, where we use the notation $\mu (a \otimes b) = a * b$, for $a, b \in A$. Then $(A, \beta)$ is a Lie $H$-pseudoalgebra, where the skew-symmetric pseudobracket $\beta$ is given by
\begin{align*}
    \beta (a \otimes b) = [a * b] := a * b - (\sigma_{12} \otimes_H 1) (b * a), \text{ for } a, b \in A.
\end{align*}

In the following, we will generalize this result in the strongly homotopy context.

\begin{thm}
    Let $(\mathcal{A}, \{ \mu_k \}_{k \geq 1})$ be an $A_\infty$ $H$-pseudoalgebra. Then $(\mathcal{A}, \{ \beta_k \}_{k \geq 1})$ is an $L_\infty$ $H$-pseudoalgebra, where
    \begin{align}\label{beta-map}
        \beta_k (a_1, \ldots, a_k) = \sum_{\sigma \in S_k} (-1)^\sigma \epsilon(\sigma) (\sigma \otimes_H 1) (\mu_k) (a_{\sigma^{-1} (1)}, \ldots, a_{\sigma^{-1} (k)} ). 
    \end{align}
\end{thm}

\begin{proof}
    Let $\mathcal{W} = \mathcal{A}[-1]$. Since $(\mathcal{A}, \{ \mu_k \}_{k \geq 1})$ is ab $A_\infty$ $H$-pseudoalgebra, there is a Maurer-Cartan element in the graded Lie algebra $(\oplus_{p \in \mathbb{Z}} \mathrm{Hom}^p_\mathrm{poly} (\mathcal{W}, \mathcal{W}), \llbracket ~, ~ \rrbracket^\sim)$. The Maurer-Cartan element is precisely given by $\nu = \sum_{k \geq 1} \nu_k \in \mathrm{Hom}^{-1}_\mathrm{poly} (\mathcal{W}, \mathcal{W})$, where
    \begin{align*}
        \nu_k = (-1)^{\frac{k (k-1)}{2}}~ s \circ \mu_k \circ (s^{-1})^{\otimes k}.
    \end{align*}
    Next, consider the collection of skew-symmetric polylinear maps $\{ \beta_k \}_{k \geq 1}$ as defined in (\ref{beta-map}). If $\{ \eta_k \}_{k \geq 1}$ are the symmetric polylinear maps in $\mathrm{symHom}^{-1}_{\mathrm{poly}} (\mathcal{W}, \mathcal{W})$ defined by $\eta_k = (-1)^{\frac{k (k-1)}{2}} ~ s \circ \beta_k \circ (s^{-1})^{\otimes k}$, then by a simple observation (see also \cite{wu}) shows that
    \begin{align*}
        (\nu_k)^\mathrm{sym} = \eta_k, \text{ for all } k \geq 1.
    \end{align*}
    Hence $(\sum_{k \geq 1} \nu_k)^{\mathrm{sym}} = \sum_{k \geq 1} (\nu_k)^{\mathrm{sym}} = \sum_{k \geq 1} eta_k$. Therefore,
    \begin{align*}
        \llbracket \sum_{k \geq 1} \eta_k, \sum_{k \geq 1} \eta_k \rrbracket = \llbracket  (\sum_{k \geq 1} \nu_k)^{\mathrm{sym}}, (\sum_{k \geq 1} \nu_k)^{\mathrm{sym}}   \rrbracket 
        = \big(  \llbracket \sum_{k \geq 1} \nu_k, \sum_{k \geq 1} \nu_k \rrbracket^\sim \big) =  0. 
    \end{align*}
    The result now follows as a consequence of Theorem \ref{mc-thm}.
\end{proof}

\section{Skeletal and strict homotopy Lie {\em H}-pseudoalgebras}\label{sec5}
Our aim in this section is to discuss and study some particular classes of $L_\infty$ $H$-pseudoalgebras. We first start with those $L_\infty$  $H$-pseudoalgebras whose underlying graded left $H$-module is concentrated in arities $0$ and $1$. We call them $2$-term $L_\infty$  $H$-pseudoalgebras. Next, we consider skeletal and strict $L_\infty$  $H$-pseudoalgebras that are special classes of $2$-term $L_\infty$ $H$-pseudoalgebras. We show that skeletal $L_\infty$ $H$-pseudoalgebras correspond to third cocycles of (ordinary) Lie $H$-pseudoalgebras. Finally, we introduce crossed modules of Lie $H$-pseudoalgebras and show that they correspond to strict $L_\infty$ $H$-pseudoalgebras. Our results in this section are motivated by the results developed by Baez and Crans \cite{baez-crans} for classical $L_\infty$ algebras.

We have already mentioned that $2$-term $L_\infty$ $H$-pseudoalgebras are those $L_\infty$ $H$-pseudoalgebras whose underlying graded left $H$-module is concentrated in arities $0$ and $1$. Hence by assuming $\mathcal{L} = \mathcal{L}_0 \oplus \mathcal{L}_1$ in Definition \ref{shlieH}, we get the following.

\begin{definition}\label{defn2term}
A {\em $2$-term $L_\infty$ $H$-pseudoalgebra} is a triple $(\mathcal{L}_1 \xrightarrow{\beta_1} \mathcal{L}_0, \beta_2, \beta_3)$ consisting of a $2$-term chain complex $\mathcal{L}_1 \xrightarrow{\beta_1} \mathcal{L}_0$ of left $H$-modules, a skew-symmetric polylinear map $\beta_2 \in \mathrm{Hom}_{H^{\otimes 2}} (\mathcal{L}_i \boxtimes \mathcal{L}_j, H^{\otimes 2} \otimes_H \mathcal{L}_{i+j})$ and another skew-symmetric polylinear map $\beta_3 \in \mathrm{Hom}_{H^{\otimes 3}} (\mathcal{L}_0 \boxtimes \mathcal{L}_0 \boxtimes \mathcal{L}_0, H^{\otimes 3} \otimes_H \mathcal{L}_{1})$ subject to satisfy the following conditions:

\medskip

(i) $\beta_1 (\beta_2 (x, u)) = \beta_2 (x, \beta_1 (u)),$

\medskip

(ii) $\beta_2 (\beta_1 (u), v) = \beta_2 (u, \beta_1 (v))$, 

\medskip

(iii) $\beta_2 (   \beta_2 (x, y), z) - \beta_2 (x, \beta_2 (y, z)) + \beta_2 (y, \beta_2 (x, z)) =  - \beta_1 (\beta_3 (x, y, z)),$

\medskip

(iv) $\beta_2 (   \beta_2 (x, y), v) - \beta_2 (x, \beta_2 (y, v)) + \beta_2 (y, \beta_2 (x, v)) = - \beta_3 (x, y, \beta_1 (v)),$

\medskip

(v) (the Jacobiator identity) 
\begin{align*}
    \beta_2 ( x, & \beta_3 (y, z, w)) - \beta_2 (y, \beta_3 (x, z, w)) + \beta_2 (z, \beta_3 (x, y, w)) - \beta_2 (w, \beta_3 (x, y, z)) \\
   & - \beta_3 ( \beta_2 (x, y), z, w) + \beta_3 ( \beta_2 (x, z), y, w) - \beta_3 ( \beta_2 (x, w), y, z) \\
   & - \beta_3 ( \beta_2 (y, z), x, w) + \beta_3 (\beta_2 (y, w), x, z) - \beta_3 (\beta_2 (z, w), x, y) = 0,
\end{align*}
for all $x, y, z, w \in \mathcal{L}_0$ and $u, v \in \mathcal{L}_1$.
\end{definition}

Let $\mathcal{L} = (\mathcal{L}_1 \xrightarrow{\beta_1} \mathcal{L}_0, \beta_2, \beta_3)$ and $\mathcal{L}' = (\mathcal{L}'_1 \xrightarrow{\beta'_1} \mathcal{L}'_0, \beta'_2, \beta'_3)$ be $2$-term $L_\infty$ $H$-pseudoalgebras. A {\em morphism} between them is given by a triple $f = (f_0, f_1, f_2)$ consisting of $H$-linear maps $f_0: \mathcal{L}_0 \rightarrow \mathcal{L}_0'$ and $f_1: \mathcal{L}_1 \rightarrow \mathcal{L}_1'$, and a polylinear map $f_2 \in \mathrm{Hom}_{H^{\otimes 2}} (\mathcal{L}_0 \boxtimes \mathcal{L}_0, H^{\otimes 2} \otimes_H \mathcal{L}_1')$ satisfying
\begin{align*}
    \beta_1 \circ f_0 =~& f_0 \circ \beta_1',\\
    ( \mathrm{id}_{H^{\otimes 2}} \otimes_H f_0) (\beta_2 (x, y)) - \beta_2' (f_0 (x), f_0 (y)) =~& ( \mathrm{id}_{H^{\otimes 2}} \otimes_H \beta_1') (f_2 (x, y)), \\
    ( \mathrm{id}_{H^{\otimes 2}} \otimes_H f_1) (\beta_2 (x, u)) - \beta_2' (f_0 (x), f_1 (u)) =~&  f_2 (x, \beta_1 (u)), \\
    \beta_2' (f_2 (x, y), f_0 (z)) - \beta_2' (f_0 (x), f_2 (y, z)) ~+~& \beta_2' (f_0 (y), f_2 (x, z)) 
     + f_2 ( \beta_2 (x, y), z) \\
    - f_2 (x, \beta_2 (y, z)) + f_2 (y, \beta_2 (x, z)) =~& \beta_3' (f_0 (x), f_0 (y), f_0 (z)) - ( \mathrm{id}_{H^{\otimes 3}} \otimes_H f_1) (\beta_3 (x, y, z)),
\end{align*}
for all $x, y, z \in \mathcal{L}_0$ and $u \in \mathcal{L}_1$. We often denote a morphism as above by $(f_0, f_1, f_2) : \mathcal{L} \rightsquigarrow \mathcal{L}'$.

Let $\mathcal{L} = (\mathcal{L}_1 \xrightarrow{\beta_1} \mathcal{L}_0, \beta_2, \beta_3)$, $\mathcal{L}' = (\mathcal{L}'_1 \xrightarrow{\beta'_1} \mathcal{L}'_0, \beta'_2, \beta'_3)$ and $\mathcal{L}'' = (\mathcal{L}''_1 \xrightarrow{\beta''_1} \mathcal{L}''_0, \beta''_2, \beta''_3)$ be $2$-term $L_\infty$ $H$-pseudoalgebras. Suppose $f = (f_0, f_1, f_2) : \mathcal{L} \rightsquigarrow \mathcal{L}'$ and $g = (g_0, g_1, g_2) : \mathcal{L}' \rightsquigarrow \mathcal{L}''$ are morphisms between $2$-term $L_\infty$ $H$-pseudoalgebras. Then their composition $g \circ f : \mathcal{L} \rightsquigarrow \mathcal{L}''$ is given by 
\begin{align*}
g \circ f  = (g_0 \circ f_0, g_1 \circ f_1, g_2 \circ (f_0 \otimes f_0) + (\mathrm{id}_{H^{\otimes 2}} \otimes g_1) \circ f_2 ).
\end{align*}

With all the above definitions and notations, we have the following result.

\begin{proposition}
    The collection of all $2$-term $L_\infty$ $H$-pseudoalgebras and morphisms between them forms a category. (We denote this category by {\bf 2shLie}$H$).
\end{proposition}

\medskip

In the following, we consider skeletal and strict $L_\infty$ $H$-pseudoalgebras and provide their characterizations.

\begin{definition}
    Let $(\mathcal{L}_1 \xrightarrow{\beta_1} \mathcal{L}_0, \beta_2, \beta_3)$ be a $2$-term $L_\infty$ $H$-pseudoalgebra. It is said to be 

    (i) {\em skeletal} if $\beta_1 = 0$,

    (ii) {\em strict} if $\beta_3  = 0$.
\end{definition}

Let $(\mathcal{L}_1 \xrightarrow{ 0} \mathcal{L}_1, \beta_2, \beta_3)$ be a skeletal $L_\infty$ $H$-pseudoalgebra. It follows from condition (iii) of the Definition \ref{defn2term} that the left $H$-module $\mathcal{L}_0$ with the skew-symmetric polylinear map $\beta_2 \in \mathrm{Hom}_{H^{\otimes 2}} (\mathcal{L}_0 \boxtimes \mathcal{L}_0, H^{\otimes 2} \otimes_H \mathcal{L}_0)$ is a Lie $H$-pseudoalgebra. Further, it follows from the condition (iv) of the same definition that the left $H$-module $\mathcal{L}_1$ equipped with the polylinear map $\beta_2 \in \mathrm{Hom}_{H^{\otimes 2}} (\mathcal{L}_0 \boxtimes \mathcal{L}_1, H^{\otimes 2} \otimes_H \mathcal{L}_1)$ is a representation of the Lie $H$-pseudoalgebra $(\mathcal{L}_0, \beta_2)$. Finally, the condition (v) is equivalent to $(\delta \beta_3) (x, y, z, w) = 0$, where $\delta$ is the coboundary map of the Lie $H$-pseudoalgebra $(\mathcal{L}_0, \beta_0)$ with coefficients in the representation $(\mathcal{L}_1, \beta_2)$. In other words, $\beta_3 \in C^3_{\text{Lie-}H} (\mathcal{L}_0, \mathcal{L}_1)$ is a $3$-cocycle.

\begin{proposition}\label{skeletal-thm}
    There is a one-to-one correspondence between skeletal $L_\infty$ $H$-pseudoalgebras and triples of the form $(L, M, \theta)$ in which $L$ is a Lie $H$-pseudoalgebra, $M$ is a representation and $\theta$ is a $3$-cocycle.
\end{proposition}

\begin{proof}
Let $(\mathcal{L}_1 \xrightarrow{ 0} \mathcal{L}_1, \beta_2, \beta_3)$ be a skeletal $L_\infty$ $H$-pseudoalgebra. Then $\big(  (\mathcal{L}_0, \beta_2) ,  (\mathcal{L}_1, \beta_2) , \beta_3 \big)$ is the required triple.

Conversely, let $( L, M, \theta)$ be a triple in which $L = (L, \beta)$ is a Lie $H$-pseudoalgebra, $M = (M, \gamma)$ is a representation and $\theta$ is a $3$-cocycle. Then $(M \xrightarrow{0} L, \beta_2, \beta_3)$ is a skeletal $L_\infty$ $H$-pseudoalgebra, where 
\begin{align*}
    \beta_2 (x, y) = \beta (x \otimes y), ~~~ \beta_2 (x, u) = - \beta_2 (u, x) = \gamma (x, u) ~~~ \text{ and } ~~~ \beta_3  = \theta, \text{ for } x, y \in L, u \in M.
\end{align*}
The above two correspondences are inverses to each other.
\end{proof}

The above result has generalizations for higher cocycles. Let $n > 2$ be a fixed natural number and $(\mathcal{L}, \{ \beta_k \}_{k \geq 1})$ be an $L_\infty$ $H$-pseudoalgebra whose underlying graded left $H$-module $\mathcal{L}$ is concentrated in arities $0$ and $n-1$. That is, $\mathcal{L}$ and the map $\beta_1$ (which is automatically trivial in this case) can be realized as the chain complex $\mathcal{L}_{n-1} \rightarrow 0 \rightarrow \cdots \rightarrow 0 \rightarrow \mathcal{L}_0$. Further, for the degree reason, it can be easily observed that the only non-trivial operations in $\{ \beta_k \}_{k \geq 1}$ are given by
\begin{align*}
    \beta_2 \in \mathrm{Hom}_{H^{\otimes 2}} (\mathcal{L}_0 \boxtimes L_0,& H^{\otimes 2} \otimes_H \mathcal{L}_0), \quad \beta_2 \in \mathrm{Hom}_{H^{\otimes 2}} (\mathcal{L}_0 \boxtimes L_{n-1}, H^{\otimes 2} \otimes_H \mathcal{L}_{n-1}) ~~ \text{ and } \\
    & \beta_{n+1} \in \mathrm{Hom}_{H^{\otimes n+1}} (\mathcal{L}_0^{\boxtimes n+1}, H^{\otimes n+1} \otimes_H \mathcal{L}_{n-1}).
\end{align*}
Then from the higher Jacobi identities (for $N = 3$), it follows that the skew-symmetric polylinear map $\beta_2 \in \mathrm{Hom}_{H^{\otimes 2}} (\mathcal{L}_0 \boxtimes L_0, H^{\otimes 2} \otimes_H \mathcal{L}_0)$ makes the left $H$-module $\mathcal{L}_0$ into a Lie $H$-pseudoalgebra and the polylinear map $\beta_2 \in \mathrm{Hom}_{H^{\otimes 2}} (\mathcal{L}_0 \boxtimes L_{n-1}, H^{\otimes 2} \otimes_H \mathcal{L}_{n-1})$ makes the pair $(\mathcal{L}_{n-1}, \beta_2)$ into a representation of the Lie $H$-pseudoalgebra $(\mathcal{L}_0, \beta_2)$. Further, the higher Jacobi identity for $N = n+2$ is equivalent to $(\delta \beta_{n+1}) (x_1, \ldots, x_{n+2}) = 0$, where $\delta$ is the coboundary operator of the Lie $H$-pseudoalgebra $(\mathcal{L}_0, \beta_2)$ with coefficients in the representation $(\mathcal{L}_{n-1}, \beta_2)$. We have the following result whose proof is similar to the proof of Proposition \ref{skeletal-thm}.

\begin{proposition}
    There is a one-to-one correspondence between $L_\infty$ $H$-pseudoalgebras whose underlying graded left $H$-module is concentrated in arities $0$ and $n-1$ (for $n > 2$), and triples of the form $(L, M, \theta)$ in which $L$ is a Lie $H$-pseudoalgebra, $M$ is a representation and $\theta$ is a $(n+1)$-cocycle.
\end{proposition}

We now introduce crossed modules of Lie $H$-pseudoalgebras and find their relations with strict $L_\infty$ $H$-pseudoalgebras.

\begin{definition}
    A {\em crossed modules} of Lie $H$-pseudoalgebras is a quadruple $(L, L', \varphi, \gamma)$ in which $L$, $L'$ are both Lie $H$-pseudoalgebras, $\varphi : L' \rightarrow L$ is a morphism of Lie $H$-pseudoalgebras and $\gamma \in \mathrm{Hom}_{H^{\otimes 2}} (L \boxtimes L', H^{\otimes 2} \otimes_{H} L')$ is a polylinear map that makes $(L', \gamma)$ into a representation of the Lie $H$-pseudoalgebra $L$ satisfying additionally
    \begin{align*}
        & (\mathrm{id}_{H^{\otimes 2}} \otimes_H \varphi) ( \gamma (x, y')) = [x * \varphi (y')], ~~~~ \gamma (  \varphi (x'), y') = [x' * y']' ~~~ \text{ and } \\
        & [\gamma (x, y') * z']' = \gamma (x, [y' * z']') - [y' * \gamma (x, z')]', \text{ for } x \in L, x', y', z' \in L'.
    \end{align*}
    Here $[\cdot * \cdot]$ and $[\cdot * \cdot]'$ denote the pseudobrackets of the Lie $H$-pseudoalgebras $L$ and $L'$, respectively.
\end{definition}

Let $L$ be a Lie $H$-pseudoalgebra with the pseudobracket $\beta = [\cdot * \cdot]$. Then the quadruple $(L, L, \mathrm{id}, \beta)$ is a crossed modules of Lie $H$-pseudoalgebras. 

Let $L$ be a Lie $H$-pseudoalgebra and $N \subset L$ be an ideal (i.e. $[x * n] \in H^{\otimes 2} \otimes_H N$, for all $x \in L$ and $n \in N$). Then the quadruple $(L, N, i, ad)$ is a crossed modules of Lie $H$-pseudoalgebras, where $i : N \hookrightarrow L$ is the inclusion map and $ad$ denotes the adjoint representation.

Let $L, L'$ be two Lie $H$-pseudoalgebras and $\theta : L \rightarrow L'$ be a morphism of Lie $H$-pseudoalgebras. Then $(L, \mathrm{ker ~}\theta, i, ad)$ is a crossed modules of Lie $H$-pseudoalgebras.

The following result is straightforward, hence we omit the proof.

\begin{proposition}\label{strict-prop}
  Let $(L, L', \varphi, \gamma)$ be a crossed modules of Lie $H$-pseudoalgebras. Then the direct sum left $H$-module $L \oplus L'$ carries a Lie $H$-pseudoalgebra structure with the pseudobracket
  \begin{align*}
      [(x, x') * (y, y')]_{L \oplus L'} := \big(   [x * y], \gamma (x, y') - (\sigma_{12} \otimes_H 1) \gamma (y, x') + [x' * y']'  \big), \text{ for } (x,x'), (y,y') \in L \oplus L'.
  \end{align*}
\end{proposition}

\begin{thm}\label{strict-thm}
    There is a one-to-one correspondence between strict $L_\infty$ $H$-pseudoalgebras and crossed modules of Lie $H$-pseudoalgebras.
\end{thm}

\begin{proof}
    Let $(\mathcal{L}_1 \xrightarrow{ \beta_1} \mathcal{L}_0, \beta_2, \beta_3 = 0)$ be a strict $L_\infty$ $H$-pseudoalgebra. Then it follows from condition (iii) of Definition \ref{defn2term} that the left $H$-module $\mathcal{L}_0$ equipped with the skew-symmetric polylinear map $\beta_2 \in \mathrm{Hom}_{H^{\otimes 2}} (\mathcal{L}_0 \boxtimes \mathcal{L}_0 , H^{\otimes} \otimes_H \mathcal{L}_0)$ is a Lie $H$-pseudoalgebra. Next, consider the left $H$-module $\mathcal{L}_1$ with the skew-symmetric polylinear map $[u * v]' := \beta_2 ( \beta_1 (u), v) = \beta_2 (u, \beta_1 (v))$, for $u, v \in \mathcal{L}_1$. Then it follows from condition (iv) of Definition \ref{defn2term} that $\mathcal{L}_1$ is a Lie $H$-pseudoalgebra. Further, the $H$-linear map $\beta_1 : \mathcal{L}_1 \rightarrow \mathcal{L}_0$ is a morphism of Lie $H$-pseudoalgebras (follows from condition (i)). Finally, the polylinear map $\beta_2 \in \mathrm{Hom}_{H^{\otimes 2}} (\mathcal{L}_0 \boxtimes \mathcal{L}_1 , H^{\otimes 2} \otimes_H \mathcal{L}_1)$ defines a representation of the Lie $H$-pseudoalgebra $(\mathcal{L}_0 , \beta_2)$ on the left $H$-module $\mathcal{L}_1$. With these notations, it is easy to verify that $(\mathcal{L}_0, \mathcal{L}_1, \beta_1, \beta_2)$ is a crossed modules of Lie $H$-pseudoalgebras.

    On the other hand, if $( (L, [\cdot * \cdot]), (L', [\cdot * \cdot]'), \varphi, \gamma)$ is a crossed modules of Lie $H$-pseudoalgebras, then it can be verified that $(L' \xrightarrow{\varphi} L, \beta_2, \beta_3  = 0)$ is a strict $L_\infty$ $H$-pseudoalgebra, where
    \begin{align*}
        \beta_2 (x, y) = [x * y], ~~~~ \beta_2 (x, x') = - \beta_2 (x', x) := \gamma (x, x') ~~~ \text{ and } ~~~ \beta_2 (x', y') = 0, \text{ for } x, y \in L, x', y' \in L'.
    \end{align*}
    Finally, the above two correspondences are inverses to each other.
\end{proof}

Combining Proposition \ref{strict-prop} with Theorem \ref{strict-thm}, we get the following result.

\begin{proposition}
    Let $(\mathcal{L}_1 \xrightarrow{\beta_1} \mathcal{L}_0, \beta_2, \beta_3  = 0)$ be a strict $L_\infty$ $H$-pseudoalgebra. Then $(\mathcal{L}_0 \oplus \mathcal{L}_1, [\cdot * \cdot]_{\mathcal{L}_0 \oplus \mathcal{L}_1})$ is a Lie $H$-pseudoalgebra, where
    \begin{align*}
        [(x, u) * (y, v)]_{ \mathcal{L}_0 \oplus \mathcal{L}_1 } := \big(  \beta_2 (x, y), \beta_2 (x, v) - (\sigma_{12} \otimes_H 1) \beta (y, u) + \beta_2 (\beta_1 (u), v) \big), \text{ for } (x, u), (y, v) \in \mathcal{L}_0 \oplus \mathcal{L}_1.
    \end{align*}
\end{proposition}

\section{Categorification of Lie {\em H}-pseudoalgebras}\label{sec6}
Categorification of Lie algebras, also known as Lie-$2$ algebras was introduced by Baez and Crans \cite{baez-crans} in the study of higher structures. They observed that Lie-$2$ algebras are closely related to $L_\infty$ algebras. Among others, they showed that the category of Lie-$2$ algebras is equivalent to the category of $2$-term $L_\infty$ algebras. In this section, we generalize the above result in the context of Lie $H$-pseudoalgebras. For this, we first introduce Lie-$2$ $H$-pseudoalgebras as the categorification of Lie $H$-pseudoalgebras and show that the collection of all Lie-$2$ $H$-pseudoalgebras form a category, which we denote by {\bf Lie2$H$}. Finally, we prove that the category {\bf Lie2$H$} is equivalent to the category {\bf 2shLie}$H$ of $2$-term $L_\infty$ $H$-pseudoalgebras, hence finds a bridge between categorification and homotopification of Lie $H$-pseudoalgebras.

Let $H$ be a cocommutative Hopf algebra and $\mathcal{M}^l (H)$ be the category of left $H$-modules.

\begin{definition}
    A {\em left-$2$ module over $H$} is a category internal to the category $\mathcal{M}^l (H)$.
\end{definition}

It follows that a left-$2$ module over $H$ is a category $C= (C_1 \rightrightarrows C_0)$, whose collection of objects $C_0$ and collection of morphisms $C_1$ are both left $H$-modules such that the source and target maps $s, t : C_1 \rightarrow C_0$, the identity-assigning map $i: C_0 \rightarrow C_1$ and the composition map $C_1 \times_{C_0} C_1 \rightarrow C_1$ are all $H$-linear maps. For any $x \in C_0$, the image $i(x)$ of the identity-assigning map is often denoted by $1_x$.

Let $C = (C_1 \rightrightarrows C_0)$ and $C' = (C'_1 \rightrightarrows C'_0)$ be two left-$2$ modules over $H$. A {\em morphism} between them is a functor internal to the category $\mathcal{M}^l (H)$. Thus, a morphism from $C$ to $C'$ consists of a pair $F = (F_0, F_1)$ of $H$-linear maps $F_0 : C_0 \rightarrow C_0'$ and $F_1 : C_1 \rightarrow C_1'$ that commute with all the structure maps of the categories $C$ and $C'$. The collection of all left-$2$ modules over $H$ and morphisms between them form a category, denoted by $\mathcal{M}^{l2} (H).$

On the other hand, a {\em $2$-term chain complex over $H$} is a complex of the form $\mathcal{L}_1 \xrightarrow{\partial} \mathcal{L}_0$, where $\mathcal{L}_0$, $\mathcal{L}_1$ are both left $H$-modules and $\partial$ is a $H$-linear map. A {\em morphism} of $2$-term chain complexes from $\mathcal{L}_1 \xrightarrow{\partial} \mathcal{L}_0$ to $\mathcal{L}'_1 \xrightarrow{\partial'} \mathcal{L}'_0$ is given by a pair $(f_0, f_1)$ of $H$-linear maps $f_0 : \mathcal{L}_0 \rightarrow \mathcal{L}_0'$ and $f_1 : \mathcal{L}_1 \rightarrow \mathcal{L}_1'$ satisfying $\partial' \circ f_1 = f_0 \circ \partial$. The collection of all $2$-term chain complexes over $H$ and morphisms between them form a category, denoted by {\bf 2Term}$H$.

Let $C = (C_1 \rightrightarrows C_0)$ be a left-$2$ module over $H$, hence an object in $\mathcal{M}^{l2} (H)$. Then $\mathrm{ker}(s) \xrightarrow{ t|_{ \mathrm{ker}(s) }  } C_0$ is a $2$-term chain complex over $H$. Moreover, a morphism between left-$2$ modules over $H$ induces a morphism between the corresponding $2$-term chain complexes over $H$. Hence we obtain a functor $\mathcal{S} : \mathcal{M}^{l2} (H) \rightarrow {\bf 2Term}H$. Conversely, if $\mathcal{L}_1 \xrightarrow{ \partial} \mathcal{L}_0$ is a $2$-term chain complex over $H$ (hence an object in {\bf 2Term}$H$) then
\begin{align*}
    C:= (\mathcal{L}_0 \oplus \mathcal{L}_1 \rightrightarrows \mathcal{L}_0 )
\end{align*}
is a left-$2$ module over $H$, where the structure maps are given by
\begin{align*}
    s (x, u) = x, ~~~ t(x, u) = x + \partial u ~~~ \text{ and } ~~~ i(x) = (x, 0), \text{ for } (x, u) \in \mathcal{L}_0 \oplus \mathcal{L}_1, x \in \mathcal{L}_0. 
\end{align*}
Similarly, a morphism of $2$-term chain complexes gives rise to a morphism between the corresponding left-$2$ modules over $H$. As a consequence, we get a functor $\mathcal{T} : {\bf 2Term}H \rightarrow \mathcal{M}^{l2} (H)$. Finally, it is easy to show that the functors $\mathcal{S}$ and $\mathcal{T}$ make the categories $\mathcal{M}^{l2} (H)$ and ${\bf 2Term}H$ equivalent. See \cite{baez-crans} for the similar observation when $H = {\bf k}$.

\begin{definition}
    A {\em Lie-$2$ $H$-pseudoalgebra} is a triple $(C, [\cdot * \cdot], \mathcal{J})$ that consists of a left-$2$ module $C= (C_1 \rightrightarrows C_0)$ over $H$, a $H^{\otimes 2}$-linear skew-symmetric functor $[\cdot * \cdot ] : C \otimes C \rightarrow H^{\otimes 2} \otimes_H C$ and a $H^{\otimes 3}$-linear skew-symmetric natural isomorphism (called the {\em Jacobiator})
    \begin{align*}
        \mathcal{J}_{x, y, z} : [[ x * y] * z] \rightarrow [x * [y * z]] - [y * [x * z]]
    \end{align*}
    that satisfy the following Jacobiator identity:

    \begin{align}\label{hexa-diag}
    \xymatrix{
   &  [[[x * y] * z] * w] \ar[rdd]^{\mathcal{J}_{[x * y] , z , w}} \ar[ldd]_{[  \mathcal{J}_{x, y, z} * w]} & \\
    & & \\
  \substack{[[x * [y * z]] * w] \\- [[y * [x * z]] * w]} \ar[dd]_{ \mathcal{J}_{x, [y * z], w} - \mathcal{J}_{y, [x * z], w}  } & & \substack{[[x * y]*[z * w]] \\ - [z * [[x * y]* w]]} \ar[dd]^{1 - [z * \mathcal{J}_{x, y, w}]}\\
   & & \\
   \substack{ [x * [[y * z] * w]] - [[y * z]* [x*w]] \\ - [y * [[x * z]* w]] + [[x * z] * [y * w]]} \ar[rdd]_{  [ x * \mathcal{J}_{y, z, w}] - \mathcal{J}_{y, z, [x * w]} - [y * \mathcal{J}_{x, z, w}] + \mathcal{J}_{x, z, [y* w]}   } & & \substack{ [[x* y] * [z * w]] - [z * [x * [y * w]]] \\ + [z * [y * [x * w]]] } \ar[ldd]^{ \mathcal{J}_{x, y, [z * w]} - 1 + 1 } \\
    & & \\
    & \substack{[x * [y * [z * w]]]  - [y * [x * [z * w]]] \\
    - [z * [x * [y * w]]] + [z * [ y * [x * w]]].} & 
    }
    \end{align}
\end{definition}

Let $(C, [\cdot * \cdot], \mathcal{J})$ and $(C', [\cdot * \cdot]', \mathcal{J}')$ be two Lie-$2$ $H$-pseudoalgebras. A {\em morphism} between them is a functor $F = (F_0, F_1) : C \rightarrow C'$ between the underlying left-$2$ modules over $H$ and a natural isomorphism
\begin{align*}
    (F_2)^{x, y} : [F_0 (x) * F_0 (y)]' \rightarrow (\mathrm{id}_{H^{\otimes 2}} \otimes_H F_0) [x * y]
\end{align*}
such that for any $x, y, z \in C_0$, the following diagram is commutative
\[
\xymatrix{
[[F_0 (x) * F_0 (y)]' * F_0 (z)]' \ar[r]^{\mathcal{J}'_{F_0 (x), F_0 (y), F_0 (z)}} \ar[d]_{  [(F_2)^{x,y} * F_0 (z)]'  } & [F_0 (x) * [F_0 (y) * F_0 (z)]' ]' -  [F_0 (y) * [F_0 (x) * F_0 (z)]' ]' \ar[d]^{[F_0 (x) * (F_2)^{y, z}]' - [F_0 (y) * (F_2)^{x, z}]'}   \\
[(\mathrm{id}_{H^{\otimes 2}} \otimes_H F_0) [x * y] * F_0 (z)]' \ar[d]_{ (F_2)^{[x * y], z} } & [F_0 (x) * (\mathrm{id}_{H^{\otimes 2}} \otimes_H F_0) [y * z]]' - [ F_0 (y) * (\mathrm{id}_{H^{\otimes 2}} \otimes_H F_0) [x* z]]' \ar[d]^{ (F_2)^{x, [y * z]} - (F_2)^{y, [x * z]}  } \\
(\mathrm{id}_{H^{\otimes 3}} \otimes_H F_0) [ [x * y] * z] \ar[r]_{(\mathrm{id}_{H^{\otimes 3}} \otimes_H F_0) (\mathcal{J}_{x, y, z})}  & (\mathrm{id}_{H^{\otimes 3}} \otimes_H F_0) \big( [x * [y * z]] - [y * [x * z]]  \big).
}
\]
We often denote a morphism as above by the triple $(F_0, F_1, F_2)$.

\begin{thm}
    The collection of Lie-$2$ $H$-pseudoalgebras and morphisms between them form a category. We denote this category by {\bf Lie2$H$}.
\end{thm}

\begin{proof}
    The verification of this result is straightforward once we define the composition of morphisms and the identity morphisms. Let $(C, [\cdot * \cdot], \mathcal{J})$ and $(C', [\cdot * \cdot]', \mathcal{J}')$ be two Lie-$2$ $H$-pseudoalgebras and $(F_0, F_1, F_2)$ be a morphism from $C$ to $C'$. Suppose $(C'', [\cdot * \cdot]'', \mathcal{J}'')$ is another Lie-$2$ $H$-pseudoalgebra and $(G_0, G_1, G_2)$ is a morphism from $C'$ to $C''$. Then their composition is defined by the triple
    \begin{align*}
        \big(  G_0 \circ F_0, G_1 \circ F_1, (\mathrm{id}_{H^{\otimes 2}} \otimes_H G_0)(F_2^{-, -}) \circ (G_2)^{F_0 (-), F_0 (-)}  \big),
    \end{align*}
    where the last component can be understood by the following composition
    \begin{align*}
    [(G_0 \circ F_0)(x) * (G_0 \circ F_0)(y)]''  \xrightarrow{(G_2)^{F_0 (x), F_0 (y)}} &
    (\mathrm{id}_{H^{\otimes 2}} \otimes_H G_0) [F_0 (x) * F_0 (y)]' \\ & \xrightarrow{(\mathrm{id}_{H^{\otimes 2}} \otimes_H G_0) (F_2)^{x, y} }
    (\mathrm{id}_{H^{\otimes 2}} \otimes_H G_0 \circ F_0) [x * y].
    \end{align*}
    Moreover, if $(C, [\cdot * \cdot], \mathcal{J})$ is any Lie-$2$ $H$-pseudoalgebra, then the identity morphism on $C$ is given by the triple $(\mathrm{id}_{C_0}, \mathrm{id}_{C_1}, \mathrm{id}^{-, -})$, where $\mathrm{id}^{x, y} : [x * y] \rightarrow [x * y]$ is the identity natural isomorphism.
\end{proof}

We are now in a position to prove the main result of this section.

\begin{thm}\label{last-thm}
    The categories {\bf Lie2$H$} and {\bf 2shLie$H$} are equivalent.
\end{thm}

\begin{proof}
    Let $ ( C= (C_1 \rightrightarrows C_0), [\cdot * \cdot], \mathcal{J})$ be a Lie-$2$ $H$-pseudoalgebra. First consider the $2$-term chain complex $\mathrm{ker}(s) \xrightarrow{ t|_{\mathrm{ker}(s)} } C_0$ over $H$. We define $H^{\otimes 2}$-linear maps (both denoted by the same notation)
    \begin{align*}
        \beta_2 : C_0 \otimes C_0 \rightarrow H^{\otimes 2} \otimes_H C_0  ~~~  \text{ and } ~~~ \beta_2 : C_0 \otimes \mathrm{ker} (s) \rightarrow H^{\otimes 2} \otimes_H \mathrm{ker}(s)
    \end{align*}
    by $\beta_2 (x , y) = [x * y]$ and $\beta_2 (x, u) = [1_x * u]$, for $x, y \in C_0$ and $u \in \mathrm{ker} (s)$. Another skew-symmetric $H^{\otimes 3}$-linear map $\beta_3 : C_0 \otimes C_0 \otimes C_0 \rightarrow H^{\otimes 3} \otimes_H \mathrm{ker}(s)$ we define by
    \begin{align*}
        \beta_3 (x, y, z) = pr_2 \big(  \mathcal{J}_{x, y, z} : [[x * y] * z] \rightarrow [ x * [y * z]] - [y * [x * z]]  \big), \text{ for } x, y, z \in C_0.
    \end{align*}
    Here $pr_2 : H^{\otimes 3} \otimes_H (C_0 \oplus C_1) \rightarrow H^{\otimes 3} \otimes_H C_1$ is the projection onto the second factor. Then similar to the work of Baez and Crans \cite{baez-crans}, it can be checked that the triple $\mathcal{L} = (  \mathrm{ker}(s) \xrightarrow{ t|_{\mathrm{ker}(s)} } C_0, \beta_2, \beta_3)$ is a $2$-term $L_\infty$ $H$-pseudoalgebra.

    Next, let $( C = (C_1 \rightrightarrows C_0), [\cdot * \cdot], \mathcal{J})$ and $(C' = (C'_1 \rightrightarrows C'_0), [\cdot * \cdot]', \mathcal{J}')$ be two Lie-$2$ $H$-pseudoalgebras and $(F_0, F_1, F_2)$ be a morphism between them. Consider the corresponding $2$-term $L_\infty$ $H$-pseudoalgebras $\mathcal{L} = (  \mathrm{ker}(s) \xrightarrow{ t|_{\mathrm{ker}(s)} } C_0, \beta_2, \beta_3)$ and $\mathcal{L}' = (  \mathrm{ker}(s') \xrightarrow{ t'|_{\mathrm{ker}(s')} } C'_0, \beta'_2, \beta'_3)$. We define two $H$-linear maps $f_0 : C_0 \rightarrow C_0'$ and $f_1 : \mathrm{ker}(s) \rightarrow \mathrm{ker}(s')$, and a $H^{\otimes 2}$-linear map $f_2 : C_0 \otimes C_0 \rightarrow H^{\otimes 2} \otimes_H \mathrm{ker} (s')$ by
    \begin{align*}
        f_0 (x) = F_0 (x), ~~~~ f_1 (u) = F_1 (u) ~~~ \text{ and } ~~~ f_2 (x, y) = (F_2)^{x, y} - 1_{s   ( (F_2)^{x, y} )} = pr_2 ( (F_2)^{x, y}),
    \end{align*}
    for $x, y \in C_0$ and $u \in \mathrm{ker} (s)$. Then a straightforward observation shows that $(f_0, f_1, f_2): \mathcal{L} \rightsquigarrow \mathcal{L}'$ is a morphism of $2$-term $L_\infty$ $H$-pseudoalgebras. As a consequence, we obtain a functor $\mathcal{S} : {\bf Lie2}H \rightarrow {\bf 2shLie}H$.

    In the following, we construct a functor $\mathcal{T} :  {\bf 2shLie}H \rightarrow {\bf Lie2}H$. We start with a $2$-term $L_\infty$ $H$-pseudoalgebra $(\mathcal{L}_1 \xrightarrow{\beta_1} \mathcal{L}_0, \beta_2, \beta_3)$. First, we consider the left-$2$ module $C= (\mathcal{L}_0 \oplus \mathcal{L}_1 \rightrightarrows \mathcal{L}_0)$ over $H$. We define a $H^{\otimes 2}$-linear skew-symmetric functor $[\cdot * \cdot] : C \otimes C \rightarrow H^{\otimes 2} \otimes_H C$ by
    \begin{align*}
        [(x, u) * (y, v)] := \big(  \beta_2 (x, y), \beta_2 (x, v) - (\sigma_{12} \otimes_H 1) \beta_2 (y, v) + \beta_2 ( \beta_1 (u), v)  \big),
    \end{align*}
    for $(x, u), (y, v) \in \mathcal{L}_0 \oplus \mathcal{L}_1$. We define the Jacobiator by
    \begin{align*}
        \mathcal{J}_{x, y, z}= \big(   \beta_2 (    \beta_2 (x, y), z ), ~ \beta_3 (x, y, z)  \big), \text{ for } x, y, z \in \mathcal{L}_0.
    \end{align*}
    It follows from condition (iii) of Definition \ref{defn2term} that the Jacobiator $\mathcal{J}_{x, y, z}$ has the source $[[x * y] * z]$ and the target $[x * [y * z]] - [y * [x * z]]$, for $x, y, z \in \mathcal{L}_0$. Note that the Jacobiator is also a natural isomorphism. Finally, a direct calculation (using the conditions of Definition \ref{defn2term}) yields that the Jacobiator makes the diagram (\ref{hexa-diag}) commutative. See \cite{baez-crans} for similar calculation. Hence we obtain a Lie-$2$ $H$-pseudoalgebra  $(\mathcal{L}_0 \oplus \mathcal{L}_1 \rightrightarrows \mathcal{L}_0, [\cdot * \cdot], \mathcal{J})$.

    Next, let $\mathcal{L} = (\mathcal{L}_1 \xrightarrow{ \beta_1} \mathcal{L}_0, \beta_2, \beta_3)$ and $\mathcal{L}' = (\mathcal{L}'_1 \xrightarrow{ \beta'_1} \mathcal{L}'_0, \beta'_2, \beta'_3)$ be $2$-term $L_\infty$ $H$-pseudoalgebras and $(f_0, f_1, f_2) : \mathcal{L} \rightsquigarrow \mathcal{L}'$ be a morphism between them. We consider the corresponding Lie-$2$ $H$-pseudoalgebras
    \begin{align*}
        C = ( \mathcal{L}_0 \oplus  \mathcal{L}_1 \rightrightarrows \mathcal{L}_0, [\cdot * \cdot], \mathcal{J}) ~~ \text{ and } ~~ C' = ( \mathcal{L}'_0 \oplus  \mathcal{L}'_1 \rightrightarrows \mathcal{L}'_0, [\cdot * \cdot]', \mathcal{J}'), \text{respectively}.
    \end{align*}
    We define $H$-linear maps $F_0 : \mathcal{L}_0 \rightarrow \mathcal{L}_0'$, $F_1 : \mathcal{L}_0 \oplus \mathcal{L}_1 \rightarrow \mathcal{L}_0' \oplus \mathcal{L}_1'$ and a natural isomorphism $(F_2)^{x, y} : [F_0 (x) * F_0 (y)]' \rightarrow (\mathrm{id}_{H^{\otimes 2}} \otimes_H F_0) [x * y]$ by
    \begin{align*}
        F_0 = f_0, ~~~~ F_1 = f_0 \oplus f_1 ~~~ \text{ and } ~~ (F_2)^{x, y} = \big(  [F_0 (x) * F_0 (y)]', f_2 (x, y) \big).
    \end{align*}
    Then $(F_0, F_1, F_2)$ turns out to be a morphism of Lie-$2$ $H$-pseudoalgebras from $C$ to $C'$. It follows that we get a functor $\mathcal{T} : {\bf 2shLie}H \rightarrow {\bf Lie2}H$.

    Finally, we will show that the functors $\mathcal{S}$ and $\mathcal{T}$ make the categories $ {\bf Lie2}H$ and ${\bf 2shLie}H$ equivalent. In other words, we construct natural isomorphisms $\Lambda : \mathcal{T} \circ \mathcal{S} \Rightarrow 1_{ {\bf Lie2}H}$ and $\Upsilon : \mathcal{S} \circ \mathcal{T} \Rightarrow 1_{ {\bf 2shLie}H }$. Let $(C = (C_1 \rightrightarrows C_0), [ \cdot * \cdot], \mathcal{J})$ be a Lie-$2$ $H$-pseudoalgebra. If we apply the functor $\mathcal{S}$, we get the $2$-term $L_\infty$ $H$-pseudoalgebra $(\mathrm{ker}(s) \xrightarrow{ t|_{ \mathrm{ker}(s)}  } C_0, \beta_2, \beta_3)$. Again, if we apply the functor $\mathcal{T}$, we get the Lie-$2$ $H$-pseudoalgebra
    \begin{align*}
        (C'= (C_0 \oplus \mathrm{ker}(s) \rightrightarrows C_0), [\cdot * \cdot]', \mathcal{J}').
    \end{align*}
    Then it is easy to see that the triple $(\Lambda_0 = \mathrm{id}_{C_0}, \Lambda_1, \Lambda_2)$, where $\Lambda_1 : C_0 \oplus \mathrm{ker} (s) \rightarrow C_1$, $(\Lambda_1) (x, u) = 1_x + u$ and $(\Lambda_2)^{x, y} = \mathrm{id} : [\Lambda_0 (x) * \Lambda_0 (y)] = [x * y] \rightarrow (\mathrm{id}_{H^{\otimes 2}} \otimes_H \Lambda_0) [x * y]' = [x * y]$ is an isomorphism of Lie-$2$ $H$-pseudoalgebras from $C'$ to $C$. This construction yields the natural isomorphism $\Lambda : \mathcal{T} \circ \mathcal{S} \Rightarrow 1_{ {\bf Lie2}H}$. On the other hand, a direct observation shows that the composition functor $\mathcal{S} \circ \mathcal{T}$ is the identity functor $1_{ {\bf 2shLie}H }$. Hence we can take $\Upsilon : \mathcal{S} \circ \mathcal{T} \Rightarrow 1_{ {\bf 2shLie}H }$ to be the identity natural isomorphism. This completes the proof.
\end{proof}

\noindent  {\bf Acknowledgements.} The author would like to thank the Department of Mathematics, IIT Kharagpur for providing the beautiful academic atmosphere where the research has been carried out.
%\vspace*{1cm}

\medskip

\noindent {\bf Data Availability Statement.} Data sharing does not apply to this article as no new data were created or analyzed in this study.

\end{document}